\documentclass[11pt]{article}
\usepackage[a4paper,margin=27mm]{geometry}
\usepackage{amsmath,amssymb,amsthm,mathtools}
\usepackage{bbm}
\usepackage{xcolor}
\usepackage{needspace}
\usepackage{tikz}
\usetikzlibrary{arrows.meta,patterns}
\usepackage{placeins}
\usepackage[utf8]{inputenc}
\usepackage[T1]{fontenc}
\usepackage[colorlinks=true,linkcolor=blue,citecolor=blue,urlcolor=blue]{hyperref}
\newtheorem{theorem}{Theorem}[section]
\newtheorem{lemma}[theorem]{Lemma}
\newtheorem{problem}[theorem]{Problem}
\newtheorem{proposition}[theorem]{Proposition}

\newtheorem{claim}[theorem]{Claim}
\newtheorem*{claim*}{Claim}
\newtheorem{fact}[theorem]{Fact}
\theoremstyle{definition}\newtheorem{definition}[theorem]{Definition}
\theoremstyle{remark}\newtheorem{remark}[theorem]{Remark}
\newtheorem*{remark*}{Remark}
\newcommand{\E}{\mathbb E}
\newcommand{\PP}{\mathbb P}
\newcommand{\Var}{\operatorname{Var}}
\newcommand{\Cov}{\operatorname{Cov}}
\newcommand{\cD}{\mathcal D}
\newcommand{\cF}{\mathcal F}
\newcommand{\cG}{\mathcal G}
\newcommand{\cH}{\mathcal H}
\newcommand{\cK}{\mathcal K}
\newcommand{\cA}{\mathcal A}
\newcommand{\cB}{\mathcal B}
\newcommand{\eps}{\varepsilon}
\numberwithin{equation}{section}
\allowdisplaybreaks[2]
\title{Proving the optimality of the product construction for the intersection--union problem}
\author{Fan Chang$^1$, Peter Frankl$^2$, Jian Wang$^3$\\[10pt]
$^{1}$School of Statistics and Data Science, Nankai University, Tianjin, China\\[6pt]
$^{2}$R\'{e}nyi Institute, Budapest, Hungary\\[6pt]
$^{3}$School of Mathematics, Sichuan University, Chengdu, 610065, China.\\[6pt]
E-mail: $^1$1120230060@mail.nankai.edu.cn,  $^2$frankl.peter@renyi.hu, \\ $^3$wangjianmath01@scu.edu.cn
}
\date{}
\hypersetup{pdftitle={The Intersection--Union Product Conjecture},pdfauthor={Fan Chang, Peter Frankl, Jian Wang}}
\begin{document}
\maketitle

\begin{abstract}
Let $m(n,p,q)$ be the largest size of a family of subsets of $[n]$
whose members pairwise intersect in at least $p$ vertices and omit
at least $q$ common vertices. We prove that a product on disjoint sets attains this maximum whenever $p,q\ge2$ and
$n\ge2^{15}(p+q-2)^3$.
Together with the second author's result for $p=1$ or $q=1$,
this establishes Frankl's Intersection--Union Product Conjecture in a wide range.
\end{abstract}


\section{Introduction}\label{sec:introduction}

Let $[n]=\{1,2,\ldots,n\}$, and let $2^{[n]}$ denote its power set.
A family $\cF\subseteq2^{[n]}$ is called \emph{$p$-intersecting} if
$|F\cap F'|\ge p$ for every $F,F'\in\cF$.
It is called \emph{$q$-union} if $|F\cup F'|\le n-q$ for every $F,F'\in\cF$.
A family with both properties is called a \emph{$(p,q)$-family}.

Let $\cF^c=\{[n]\setminus F\colon F\in \cF\}$ be the family of complements. It is easy to see that $\cF$ is $p$-intersecting if and only if $\cF^c$ is $p$-union.

For $U\subset [n]$ with $|U|\geq t$ and $|U|\equiv t\pmod2$, define
\[
\cK(n,U,t)=\left\{K\subset [n]\colon \left|K\cap U\right|\geq \frac{|U|+t}{2}\right\}.
\]
Clearly $\cK(n,U,t)$ is $t$-intersecting.

Since for the proofs the Boolean function analysis method is used heavily, it is convenient to use the uniform measure (the normalized size) instead of the counting measure of the families. For  $\cF\subset 2^{[n]}$, define its {\it uniform measure}  as
\[
\mu(\cF) = \frac{|\cF|}{2^n}.
\]
For $n\geq t$, define
\[
w(n,t):=\max\left\{\mu(\cF)\colon \cF\subset 2^{[n]}\mbox{ is $t$-intersecting} \right\}.
\]

Let us recall the classical Katona Theorem.

\begin{theorem}[\cite{K1}]\label{thm-katona}
Let $n\geq t>0$ be integers.
\begin{itemize}
  \item[(i)] If $n-t$  is even then
  \begin{align*}
  w(n,t)=\mu(\cK(n,[n],t))=2^{-n} \sum_{i\geq \frac{n+t}{2}} \binom{n}{i}.
  \end{align*}
  \item[(ii)] If $n-t$ is odd then
   \begin{align*}
  w(n,t)=\mu(\cK(n,[n-1],t))= 2^{-n+1}\sum_{i\geq \frac{n-1+t}{2}} \binom{n-1}{i}.
  \end{align*}
\end{itemize}
\end{theorem}

Note that for $n-t$ odd, $w(n,t)=w(n-1,t)$.

Katona subsequently asked the following question.

\begin{problem}[Katona \cite{K2}]
Let $n,p,q$ be positive integers, $n\geq p+q$. Determine or estimate
$$
w(n,p,q):=\max\left\{\mu(\cF)\colon \cF\subset 2^{[n]}\mbox{ is both $p$-intersecting and $q$-union} \right\}.
$$
\end{problem}

Taking complements shows that $w(n,p,q)=w(n,q,p)$.

The equality $w(n,1,1)=1/4$ is a well-known fact that easily follows from the Harris--Kleitman correlation inequality~\cite{Ha,Kl}. It was stated almost simultaneously by Daykin--Lov\'{a}sz~\cite{DL} and Seymour~\cite{Se} with different proofs.

Confirming the conjecture of Katona~\cite{K2}, the second author~\cite{F1} established the value of $w(n,1,q)$:
\begin{align}\label{ineq-1.4}
w(n,1,q) =\frac{1}{2}w(n-1,q).
\end{align}
Later shorter proofs were given by D.~L.~Wang~\cite{W} and by the second author~\cite{FT}.

For the general case, there is a natural product construction. Partition $[n]$ into $X$ and $Y$, take a $p$-intersecting family $\cA\subseteq2^X$ and a
$q$-union family $\cB\subseteq2^Y$, and define
$$
\cA\times\cB=\{A\cup B:A\in\cA,\ B\in\cB\}.
$$
Then $\cA\times\cB$ is both $p$-intersecting and $q$-union.

Frankl's Intersection-Union Product Conjecture, first appeared in~\cite{F1} in the $p=q=2$ case, asserts that such a construction always attains the maximum. In~\cite{FW}, the conjecture is verified for roughly  $n\leq p+q+\frac{1}{2}(p+q)^{1/3}$.

In the present paper, we prove the conjecture for $n\ge N(p,q)$,
with the following explicit bound.

\begin{theorem}\label{thm:main}
Let $n,p,q$ be positive integers with $n\ge2^{15}(p+q-2)^3$ and $p,q\geq 2$.  Then
\begin{equation}\label{eq:answer}
w(n,p,q)=\max_{n_1+n_2=n}w(n_1,p)w(n_2,q).
\end{equation}
\end{theorem}

Define
\[
B_n=\max_{n_1+n_2=n}w(n_1,p)w(n_2,q).
\]
 We call that $m$ is {\it admissible} for $t$ if $m-t$ is even. For admissible $m$ for $t$,
\[
w(m,t)-w(m-2,t) =\frac{t-1}{m-1}2^{-m}\binom{m-1}{(m-t)/2}>0.
\]
Thus $w(m,t)$ is non-decreasing in $m$.
Let $\eps\in \{0,1\}$ be defined by $\eps\equiv n-p-q\pmod2$. Since  $w(n_1,p)=w(n_1-1,p)$  for $n_1-p$ odd, by non-decreasing we infer that
\begin{equation}\label{eq:answer2}
B_n=\max_{\substack{n_1\ge p,\ n_2\ge q\\n_1\equiv p,\ n_2\equiv q\pmod 2\\n_1+n_2=n-\eps}}w(n_1,p)w(n_2,q).
\end{equation}
It implies that $B_{n+1}=B_n$ when  $n-p-q$ is even.

\begin{fact}\label{fact-reduction}
If $w(n,p,q)=B_n$ holds for all $n$ with $n-p-q$ odd and $n\geq N(p,q)$, then $w(n,p,q)=B_n$ holds as well for  all $n$ with  $n-p-q$ even and $n\geq N(p,q)$.
\end{fact}

\begin{proof}
Note that if $\cF\subset 2^{[n]}$ is a $(p,q)$-family then $\cF\times \{\{n+1\}\}$ is also  a $(p,q)$-family. It follows that
\[
w(n+1,p,q) \geq w(n,p,q).
\]
For $n\geq N(p,q)$,  we have $n+1\geq N(p,q)$ and thereby
\[
w(n,p,q) \leq w(n+1,p,q)=B_{n+1} =B_n.
\]
On the other hand, the product construction implies $w(n,p,q)\geq B_n$. Thus the fact holds.
\end{proof}

By Fact \ref{fact-reduction}, to prove Theorem \ref{thm:main} it suffices to consider the $n-p-q$ odd case.

The rest of the proof is organized as follows:
Section~\ref{sec:preliminaries} introduces the notation and preliminary
tools and gives an overview of the proof.
Section~\ref{sec:stability} establishes a stability version of Katona's
theorem and the influence estimates needed later.
In Section~\ref{sec:overlap}, we show that the indicator functions
associated with a hypothetical counterexample have small
cross-total-influence.
Section~\ref{sec:localization} uses these estimates to approximate
the corresponding families by Katona families on disjoint subsets of $[n]$.
Section~\ref{sec:completion} completes the proof of
Theorem~\ref{thm:main} by comparing the possible gain in product measure
with the covariance loss, using shadow inequalities and
Hamming neighborhood expansion.
The appendices contain the deferred proofs, numerical estimates,
and an auxiliary inequality for log-concave sequences.

\section{Preliminaries and Proof overview}\label{sec:preliminaries}

\subsection{Notation and definitions}\label{subsec:notation}

We identify a subset of $[n]$ with its characteristic vector in
$\{0,1\}^n$. Thus we identify $2^{S}$ with $\{0,1\}^S$, which is also called {\it cubes on $S$}.
For $\cF\subseteq2^{[n]}$, let
$\mathbbm{1}_{\cF}:\{0,1\}^n\to\{0,1\}$ denote its indicator function. That is, $\mathbbm{1}_{\cF}(F)=1$ if  $F\in \cF$ and  $\mathbbm{1}_{\cF}(F)=0$ otherwise. Since $\cF$ is identified with the function $\mathbbm{1}_{\cF}$, a subset $S\subset [n]$ is also called a set of {\it coordinates}.

For a specified ground set $S$ of size $m$ and $\cF\subseteq2^S$,
write $\mu^m(\cF)=2^{-m}|\cF|$ for its \emph{uniform measure}.
Thus $\mu^n$ is the uniform measure on $2^{[n]}$.
We usually omit the superscript when the ground set is clear, but
retain it when comparing families on different ground sets. In particular, for $J\subseteq[n]$ and $\cF\subseteq2^J$,
$\mu^n(\cF\times2^{[n]\setminus J})=\mu^{|J|}(\cF)$.

We call $\cF\subset 2^{[n]}$ an {\it increasing family} (or a {\it up-set}) if $A\subset B\subset [n]$ and $A\in \cF$ imply $B\in \cF$, and  call  $\cF\subset 2^{[n]}$ an {\it decreasing  family} (or a {\it down-set})  if $B\subset A$ and $A\in \cF$ imply $B\in \cF$.

A function $f$ is \emph{increasing} if $f(x)\leq f(y)$ whenever
$x_i\leq y_i$ for every $i$, and \emph{decreasing} if the reverse
inequalities hold. Clearly $\cF\subset 2^{[n]}$ is increasing if and only if $\mathbbm{1}_{\cF}$ is increasing, and  $\cF\subset 2^{[n]}$ is decreasing if and only if $\mathbbm{1}_{\cF}$ is decreasing.

For a specified ground set $S$, let $A^c=S\setminus A$ for
$A\subseteq S$, and let $\cF^c=\{A^c:A\in\cF\}$.
For each integer $i$, write $\cF^{(i)}=\{F\in\cF:|F|=i\}$;
this family is empty when $i<0$ or $i>|S|$.
Thus $\mu^{|S|}(\cF^{(i)})=2^{-|S|}|\cF^{(i)}|$; in particular,
a complete level has measure $2^{-|S|}\binom{|S|}i$.
Families $\cF,\cG\subseteq2^S$ are \emph{cross $t$-intersecting}
if $|F\cap G|\ge t$ for every $F\in\cF$ and $G\in\cG$.

For $J\subseteq[n]$, the \emph{trace} of $\cF$ on $J$ is
\begin{equation}\label{eq:family-trace}
 \cF_{\mid J}=\{F\cap J:F\in\cF\}\subseteq2^J.
\end{equation}
If membership depends only on $J$, then
$\cF=\cF_{\mid J}\times2^{[n]\setminus J}$; we identify $\cF$
with this trace when working on $2^J$.
In particular, $\cF_{\mid[n-d+1,n]}$ is the trace on the last $d$ vertices.

For $S\subseteq J\subseteq[n]$, define the \emph{restricted projection}
\begin{equation}\label{eq:restricted-projection}
 \cF(S,J)=\{F\setminus J:F\in\cF,\ F\cap J=S\}\subseteq2^{[n]\setminus J}.
\end{equation}
It fixes the coordinates in $J$ to the characteristic vector of $S$
and retains the other coordinates.
We also use the standard abbreviations
$\cF(i)=\cF(\{i\},\{i\})$ and
$\cF(\bar i)=\cF(\varnothing,\{i\})$.
In particular,
$$
\mu^n(\cF)=2^{-|J|}\sum_{S\subseteq J}\mu^{n-|J|}(\cF(S,J)).
$$

\subsection{Boolean function analysis}\label{subsec:boolean}

Let $x^{\oplus i}$ be obtained from $x$ by flipping coordinate $i$, and let $X$ be a random vector uniform chosen from $\{0,1\}^n$. If $f$ is a function on $2^{[n]}$ taking values only in $\{0,1\}$, then $f$ is called {\it Boolean}. For a Boolean function $f:\{0,1\}^n\to\{0,1\}$, its \emph{influence} in coordinate $i$ is
$$
I_i[f]=\PP\bigl(f(X)\ne f(X^{\oplus i})\bigr)=\E\bigl[(f(X)-f(X^{\oplus i}))^2\bigr].
$$
Define the {\it influence on $P$} by $I_P[f]=\sum_{i\in P}I_i[f]$ and  the {\it total influence} by $I[f]=\sum_{i=1}^nI_i[f]$.
For Boolean functions $f,g:\{0,1\}^n\to\{0,1\}$, define the
\emph{cross-total-influence of $f$ and $g$} by
\begin{equation}\label{eq:cross-total-influence}
 \mathcal{I}[f,g]:=\sum_{i=1}^n I_i[f]I_i[g].
\end{equation}

For families, set $I_i[\cF]=I_i[\mathbbm{1}_{\cF}]$, and similarly for $I_P[\cF]$ and $I[\cF]$. If $\cF\subset 2^{[n]}$ is increasing, then
\begin{equation}\label{ineq-2.1}
 I_i[\cF]=\mu^{n-1}(\cF(i))-\mu^{n-1}(\cF(\bar i))
 =\frac{|\cF(i)|-|\cF(\bar i)|}{2^{n-1}}.
\end{equation}
Here $\cF(i)$ and $\cF(\bar i)$ are the restricted projections
from \eqref{eq:restricted-projection}, with coordinate $i$ fixed to
one and zero, respectively.

For a real function $f$ on $\{0,1\}^n$ and $i\in[n]$,
the \emph{discrete derivative} in coordinate $i$ is the function
$D_i f:\{0,1\}^{[n]\setminus\{i\}}\to\mathbb R$ defined by
$$
D_i f(x)
=f(x_1,\ldots,x_{i-1},1,x_{i+1},\ldots,x_n)
-f(x_1,\ldots,x_{i-1},0,x_{i+1},\ldots,x_n).
$$
If $f$ is increasing and Boolean, then $D_i f$ is Boolean and
$\E[D_i f]=I_i[f]$.

For $U\subset [n]$ with $|U|=m$, since $\cK(n,U,t)$ is increasing, by \eqref{ineq-2.1},
\begin{equation}\label{eq:threshold-influences}
I_i[\cK(n,U,t)] =\left\{
                \begin{array}{ll}
                  2^{-(m-1)}\binom{m-1}{(m-t)/2}, & i\in U; \\[5pt]
                  0, &  i\notin U.
                \end{array}
              \right.
\end{equation}
Let $\theta(m,t):=2^{-(m-1)}\binom{m-1}{(m-t)/2}$. Then
$$
I[\cK(n,U,t)] =\sum_{i\in [n]} I_i[\cK(n,U,t)] =m\,\theta(m,t)=:J(m,t).
$$
For $t\ge2$, we use the convention
$w(t-2,t)=\theta(t-2,t)=J(t-2,t)=0$.
 Note that
 \[
\cK(n,U\setminus\{i\},t-1)(i) =\cK(n,U,t)(i),\  \cK(n,U,t)(\bar{i})\subset \cK(n,U\setminus\{i\},t-1)(\bar{i})
 \]
 and
 \[
 \cK(n,U\setminus\{i\},t-1)(\bar{i})\setminus \cK(n,U,t)(i) = \left\{F\subset [n]\colon i\notin F,\ |F\cap (U\setminus \{i\})|=\frac{m+t}{2}-1\right\}.
 \]
Hence, by \eqref{eq:threshold-influences} one can check that for $t\geq 2$
\begin{equation}\label{eq:threshold-one-projection}
 \mu(\cK(n,U\setminus\{i\},t-1))
 =w(m-1,t-1)=w(m,t)+\frac{\theta(m,t)}2.
\end{equation}

For a Boolean function $f$, its {\it dual} is defined by $f^*(x)=1-f(1-x)$. The dual of an increasing Boolean function is increasing, and $I_i[f^*]=I_i[f]$.
In particular, $\mathcal{I}[f,g^*]=\mathcal{I}[f,g]$.

Define the {\it covariance} of $f,g$ by  $\Cov(f,g)=\E[fg]-\E f\,\E g$ and the {\it variance} of $f$ by $\Var(f)=\Cov(f,f)$.
For $r>0$, the {\it $L^r$ norm }  of $f$ is defined by $\|f\|_r=(\E|f|^r)^{1/r}$.

For $S\subseteq[n]$, the \emph{Fourier--Walsh character} is defined as $\chi_S(x)=(-1)^{\sum_{i\in S}x_i}$.
For any real-valued functions $f,g\colon \{0,1\}^n\to \mathbb{R}$, define the {\it inner
product} by
\[
\langle f,g\rangle=\E[fg].
\]
The characters $\{\chi_S(x)\colon S\subset [n]\}$ form an orthonormal basis with respect to the inner
product $\langle f,g\rangle=\E[fg]$. The \emph{Fourier--Walsh coefficient}
of $f$ at $S$ and the corresponding expansion are
\[
 \hat{f}(S)=\langle f,\chi_S\rangle
 =2^{-n}\sum_{x\in\{0,1\}^n}f(x)\chi_S(x).
\]
In particular, $\widehat f(\varnothing)=\E f$.

Define the {\it noise operator} by $\widehat{T_\theta f}(R)=\theta^{|R|}\hat{f}(R)$ for
$0\le\theta\le1$. For $0\le\rho,\sigma\le1$, the noise operators
preserve non-negativity, and satisfy $T_\rho T_\sigma=T_{\rho\sigma}$.

\begin{lemma}[{\cite[Lemma~3.4]{CC}}]
For any real functions $f,g$,
\begin{equation}\label{eq:covariance-integral}
\Cov(f,g)=\frac14\sum_{i=1}^n\int_0^1\langle D_i f,\,T_\rho D_i g\rangle \,d\rho.
\end{equation}
\end{lemma}

For $J\subseteq[n]$, the {\it conditional expectation}
$\E[f\mid J]$ is defined as:
\begin{equation}\label{eq:conditional-expectation}
 \E[f\mid J](x_J)
 =2^{-|J^c|}\sum_{y\in\{0,1\}^{J^c}}f(x_J,y).
\end{equation}
For a family $\cF$ and $S\subseteq J$,
$\E[\mathbbm{1}_{\cF}\mid J](S) =\mu^{n-|J|}(\cF(S,J))$. The
function $\E[\mathbbm{1}_{\cF}\mid J](x_J)$ is increasing  whenever $\cF$ is increasing.
The law of total covariance is
\begin{equation}\label{eq:total-covariance}
 \Cov(f,g)=\E[\Cov(f,g\mid X_J)]
       +\Cov(\E[f\mid J],\E[g\mid J]).
\end{equation}

\subsection{Classical theorems used in the proofs}\label{subsec:classical}

For $\cF\subseteq\binom{[n]}k$ and $0\le j\le k$, its
$j$th lower shadow is
$$
 \partial^j\cF
 =\{G\in\tbinom{[n]}{k-j}:G\subseteq F
\text{ for some }F\in\cF\}.
$$
In particular, $\partial^0\cF=\cF$ and $\partial\cF=\partial^1\cF$.
The same definition applies on any specified ground set. A useful elementary bound is the Sperner's bound~\cite{Sperner}:
\begin{equation}\label{eq:normalized-shadow}
 \frac{|\partial^j\cF|}{\binom n{k-j}}
 \ge\frac{|\cF|}{\binom nk}.
\end{equation}

\begin{theorem}[Katona's intersecting shadow theorem~\cite{K1}]
\label{thm:classical-shadow}
Let $\cF\subseteq\binom{[n]}k$ be $t$-intersecting,
where $1\le t\le k\le n$. For every integer $0\le j\le t$,
\begin{equation}\label{eq:classical-shadow}
 |\partial^j\cF|
 \ge\frac{\binom{2k-t}{k-j}}{\binom{2k-t}k}|\cF|.
\end{equation}
\end{theorem}

For a more general result cf. \cite{FW26}.

For any family $\cF\subseteq2^{[n]}$, let
$N_j(\cF)=\{G\subseteq[n]:|G\triangle F|\le j\text{ for some }F\in\cF\}$
denote its closed Hamming neighborhood of radius $j$; in particular, $N_j(\varnothing)=\varnothing$.

For distinct $A,B\in 2^{[n]}$, define the total ordering
$$
 A<_{\mathrm{simp}}B
 \mbox{ if }
 |A|<|B|
 \ \text{or}\
  |A|=|B|\ \text{and}\ \min(A\triangle B)\in A.
$$
This total order on $2^{[n]}$ is called the \emph{simplicial order}.
A family $\cF\subseteq2^{[n]}$ is called  an \emph{initial segment} of the simplical order if
$B\in\cF$ and $A<_{\mathrm{simp}}B$ imply $A\in\cF$.

\begin{theorem}[Harper's vertex-isoperimetric theorem~\cite{Harper}, cf. \cite{FF} for a short proof]
\label{thm:classical-harper}
Let $\cF\subseteq2^{[n]}$, and let $\cF_0$ be the initial segment
of the simplicial order on $2^{[n]}$ with $|\cF_0|=|\cF|$.
For every integer $j\ge0$,
\begin{equation}\label{eq:classical-harper}
 |N_j(\cF)|\ge|N_j(\cF_0)|.
\end{equation}
\end{theorem}

\begin{theorem}[Edge-isoperimetric inequality; {\cite[Proposition~1.5]{EKLbiased}}]
\label{thm:edge-isoperimetry}
If $f$ is a Boolean function with $\E f=\alpha$, then
\begin{equation}\label{eq:edge-isoperimetry}
 I[f]\ge\frac{2\alpha}{\ln2}\ln\frac1\alpha,
\end{equation}
where the right-hand side is zero when $\alpha=0$.
\end{theorem}

\begin{theorem}[Parseval's identity and the cube Poincar\'e inequality~\cite{OD}]
\label{thm:classical-poincare}
For every $f:\{0,1\}^n\to\mathbb{R}$,
\begin{align}
\E f^2&=\sum_{S\subseteq[n]}\hat{f}(S)^2,\label{eq:classical-parseval}\\
\Var(f)&\le\frac14\sum_{i=1}^n\E(D_i f)^2.\label{eq:classical-poincare}
\end{align}
In particular, if $f$ is Boolean, then $\Var(f)\le\frac14I[f]$.
\end{theorem}

\begin{theorem}[Bonami--Beckner hypercontractivity~\cite{OD}]
\label{thm:classical-hypercontractivity}
If $1<r\le u<\infty$ and
$0\le\theta\le\sqrt{(r-1)/(u-1)}$, then every $f:\{0,1\}^n\to\mathbb{R}$ satisfies
\begin{equation}\label{eq:classical-hypercontractivity}
 \|T_\theta f\|_u\le\|f\|_r.
\end{equation}
\end{theorem}

\begin{theorem}[Reverse hypercontractivity~\cite{MOS}]
\label{thm:classical-reverse}
For nonnegative functions $f,g:\{0,1\}^n\to\mathbb{R}_{\ge 0}$ and $0<\theta<1$,
\begin{equation}\label{eq:classical-reverse}
 \langle f,\, T_\theta g\rangle \ge\|f\|_{1-\theta}\|g\|_{1-\theta}.
\end{equation}
\end{theorem}

\begin{theorem}[Harris--Kleitman inequality~\cite{Ha,Kl}]\label{thm:harris}
For increasing $f,g:\{0,1\}^n\to\mathbb{R}$, $\Cov(f,g)\ge0$.
\end{theorem}

\subsection{Shifting}\label{subsec:shifting}

We use the standard shifting operators; see~\cite{EKR,F87}. For $\cF\subseteq2^{[n]}$ and $i\neq j$, define the shift
$S_{ij}(\cF)=\{S_{ij}(F):F\in\cF\}$,
where
$$S_{ij}(F)=\left\{
\begin{array}{ll}
(F\setminus\{j\})\cup\{i\}, & j\in F, i\notin F \text{ and } (F\setminus\{j\})\cup\{i\}\notin \cF; \\[5pt]
F, & \hbox{otherwise.}
\end{array}\right.
$$
It is well known (cf. \cite{F87}) that shifting preserves the  $p$-intersecting and $q$-union properties.

We say that $\cF\subset 2^{[n]}$ is {\it left-shifted} if $S_{ij}(\cF)= \cF$ for any $1\leq i<j\leq n$. We say that $\cF\subset 2^{[n]}$ is {\it right-shifted} if $S_{ji}(\cF)= \cF$ for any $1\leq i<j\leq n$.

\begin{lemma}\label{lem:closure}
For every $(p,q)$-family $\cF\subset 2^{[n]}$, there are increasing families
$\cA,\cB\subseteq2^{[n]}$ such that $\cA$ is $p$-intersecting
and left-shifted, $\cB$ is $q$-intersecting and right-shifted, and
$\cF \subset \cA\cap\cB^c$.
\end{lemma}
\begin{proof}
Since the shifting operator preserves the  $p$-intersecting and $q$-union properties, we may assume that $\cF$ is left-shifted.

Define
$$
\cF^*=\{A:\text{some }F\in\cF\text{ satisfies }F\subseteq A\},
\quad
\cF_*=\{A:\text{some }F\in\cF\text{ satisfies }A\subseteq F\}.
$$
Let $\cA=\cF^*$,
$\cD=\cF_*$ and $\cB=\cD^c$.
Then
$\cF\subseteq\cA\cap\cD$ and $\cA\cap\cD$
is a $(p,q)$-family.

Since $\cF$ is left-shifted, we infer that $\cA,\cD$ are both left-shifted. Since $\cB=\cD^c$, we infer that $\cB$ is right-shifted. Thus the lemma is proven.
\end{proof}

For an increasing left-shifted family, the influences are non-increasing:
$I_1\ge\cdots\ge I_n$. Indeed, for $i<j$ their difference is
$\E[\mathbbm{1}_{\cF}(X_{\{i,j\}^c},1,0)-\mathbbm{1}_{\cF}(X_{\{i,j\}^c},0,1)]\ge0$.
For an increasing right-shifted family the order is reversed.
Consequently, if $f,g$ are the indicators in Lemma~\ref{lem:closure},
then $\{i:I_i[f]\ge I_i[g]\}$ is an initial interval of $[n]$.

\subsection{Simplicial compression on a set of coordinates}
\label{subsec:simplicial-compression}

For an ordered ground set $X$, write $A_1<_{\mathrm{simp}}\cdots<_{\mathrm{simp}}A_{2^{|X|}}$ for
its subsets in the simplicial order.
For $0\le s\le2^{|X|}$, set
$$
 \mathcal T_X(s)=\{A_{2^{|X|}-s+1},\ldots,A_{2^{|X|}}\},
 \qquad \mathcal T_X(0)=\varnothing.
$$
Thus $\mathcal T_X(s)$ consists of the last $s$ subsets in
the simplicial order. It is increasing, and
$\mathcal T_X(s)\subseteq\mathcal T_X(s')$ whenever $s\le s'$.
Taking the complement of each member reverses the simplicial order;
hence $\mathcal T_X(s)^c$ consists of its first $s$ subsets.
Here and throughout this subsection, complements of subsets of $X$
are taken inside $X$.

\begin{definition}\label{def:simplicial-compression}
For $X\subseteq[n]$ and $\cF\subseteq2^{[n]}$, define
\begin{equation}\label{eq:simplicial-compression}
 \mathcal C_X\cF
 =\bigcup_{y\subseteq[n]\setminus X}
   \{A\cup y:A\in\mathcal T_X(|\cF(y,[n]\setminus X)|)\}.
\end{equation}
We call this the \emph{simplicial compression on $X$}.
It replaces each projection with the coordinates outside $X$ fixed
by the last equally many subsets of $X$ in the simplicial order.
\end{definition}

\begin{remark*}
The classical compression method replaces projections on a fixed
set of coordinates by initial segments of an isoperimetric order;
see Bollob\'as and Leader~\cite{BLcompression} for the general
framework. The operation in \eqref{eq:simplicial-compression} is
its reflected version: complement the coordinates in $X$, compress
to the first subsets in simplicial order, and complement those
coordinates again. We use this orientation to preserve increasingness. Our tool is precisely Harper's
Theorem~\ref{thm:classical-harper}, together with the fact that the Hamming neighborhood of a simplicial initial segment is again an initial segment; see also R\"aty~\cite[Section~1, Theorem~1 and the discussion following it]{Raty} for this formulation.
We prove the required preservation of $t$-intersection below as a
consequence of that theorem, rather than invoke a preservation
claim for arbitrary compressions. Unlike $S_{ij}$, this operation
need not preserve $|\cF\cap\binom{[n]}k|$ for each $k$.
\end{remark*}

\begin{lemma}\label{lem:simplicial-compression}
Let $\cF\subseteq2^{[n]}$ be increasing and $t$-intersecting, and
let $X\subseteq[n]$. Then $\mathcal C_X\cF$ is increasing and
$t$-intersecting. For every $y\subseteq[n]\setminus X$,
\begin{equation}\label{eq:compression-projection-size}
 |(\mathcal C_X\cF)(y,[n]\setminus X)|=|\cF(y,[n]\setminus X)|.
\end{equation}
In particular, its measure and the measure of every restricted
projection with the coordinates outside $X$ fixed are unchanged.
If $|X|$ is admissible for
$t$, then
\begin{equation}\label{eq:compression-distance}
 \mu^n(\mathcal C_X\cF\triangle\cK(n,X,t))
 \le\mu^n(\cF\triangle\cK(n,X,t)).
\end{equation}
\end{lemma}
\begin{proof}
Let $\cF_y=\cF(y,[n]\setminus X)$.
The identity \eqref{eq:compression-projection-size} follows from the definition.
If $y\subseteq y'$, then $\cF_y\subseteq\cF_{y'}$, implying
$\mathcal T_X(|\cF_y|)\subseteq\mathcal T_X(|\cF_{y'}|)$.
Together with the increasingness of each $\mathcal T_X(s)$,
this proves that $\mathcal C_X\cF$ is increasing.

For increasing families $\cA,\cB\subseteq2^X$ and $k\ge1$, we first
prove that
\begin{equation}\label{eq:cross-intersection-neighborhood}
 \mbox{ $\cA,\cB$ are cross $k$-intersecting if and only if }N_{k-1}(\cA^c)\cap\cB=\varnothing.
\end{equation}
One implication follows from
$|A^c\triangle B|=|A\cap B|+|X\setminus(A\cup B)|$.
For the converse, given $A\in\cA$ and $B\in\cB$, replace $A$ by
$A\cup B^c\in\cA$. Its union with $B$ is $X$ and its intersection
with $B$ is $A\cap B$, so the assumed distance bound \eqref{eq:cross-intersection-neighborhood} implies
$|A\cap B|\ge k$.

Suppose the conditions in \eqref{eq:cross-intersection-neighborhood}
hold. By Harper's theorem,
 \begin{align}\label{ineq-neq2.2}
 |N_{k-1}(\mathcal T_X(|\cA|)^c)|
 \le |N_{k-1}(\cA^c)|\le 2^{|X|}-|\cB|
 =2^{|X|}-|\mathcal T_X(|\cB|)|.
 \end{align}
The family $N_{k-1}(\mathcal T_X(|\cA|)^c)$ consists of the first subsets in the
simplicial order, whereas $\mathcal T_X(|\cB|)$ consists of the
last subsets. By \eqref{ineq-neq2.2}, these two
families are disjoint. By
\eqref{eq:cross-intersection-neighborhood},
$\mathcal T_X(|\cA|)$ and $\mathcal T_X(|\cB|)$ are cross
$k$-intersecting.

For each pair $y,y'\subseteq[n]\setminus X$, apply \eqref{eq:cross-intersection-neighborhood} with
$\cA=\cF_y$, $\cB=\cF_{y'}$ and $k=t-|y\cap y'|$ when $k>0$.
When $k\le0$, the coordinates in $y\cap y'$ already supply at
least $t$ common elements. This proves $t$-intersection.

Finally, let $\cK_X=\{A\subseteq X:|A|\ge(|X|+t)/2\}$.
Both $\cK_X$ and $\mathcal T_X(|\cF_y|)$ consist of the last
subsets in the same simplicial order. Therefore one contains the
other, and
$$
 \begin{aligned}
 |\mathcal T_X(|\cF_y|)\triangle\cK_X|
 =\bigl||\cF_y|-|\cK_X|\bigr|\le|\cF_y\triangle\cK_X|.
 \end{aligned}
$$
Sum over $y\subseteq[n]\setminus X$ and divide by $2^n$, we obtain
\eqref{eq:compression-distance}.
\end{proof}

\subsection{Majority projection}\label{subsec:majority}

\begin{definition}[Majority projection]\label{def:majority}
For $\cF\subseteq2^{[n]}$ and $P\subseteq[n]$, define
\begin{equation}\label{eq:majority-projection}
\Pi_P\cF=\left\{A\subseteq P:\mu^{n-|P|}(\cF(A,P))>\frac{1}{2}\right\},
\qquad
\Pi_P^*\cF=\Pi_P\cF\times2^{P^c}.
\end{equation}
We call $\Pi_P\cF$ the \emph{majority projection} of $\cF$ onto $P$.
\end{definition}

\begin{remark}\label{rem:majority-boolean}
For each fixed $A\subseteq P$, approximating the indicator of
$\cF(A,P)$ by zero gives error $\mu^{n-|P|}(\cF(A,P))$;
approximating it by one gives error $1-\mu^{n-|P|}(\cF(A,P))$.
Thus
\begin{equation}\label{eq:majority-best-approximation}
 \mu^n(\cF\triangle\Pi_P^*\cF)
 =\E_{A\subseteq P}\min\{
       \mu^{n-|P|}(\cF(A,P)),1-\mu^{n-|P|}(\cF(A,P))\}.
\end{equation}
Either value gives the same error at a tie. Assigning zero preserves
$t$-intersection, as proved below.
\end{remark}

\begin{lemma}\label{lem:projection}
Let $\cF\subseteq2^{[n]}$ be increasing and $t$-intersecting, and
let $P\subseteq[n]$. Then $\Pi_P\cF$ is increasing and
$t$-intersecting. If $\cF$ is left-shifted (respectively,
right-shifted), its majority projection has the same property on
the induced order of $P$. Moreover,
$\mu^n(\cF\triangle\Pi_P^*\cF)\le I_{P^c}[\cF]/2$.
\end{lemma}
\begin{proof}
If $A\subseteq C\subseteq P$, by increasingness
$\cF(A,P)\subseteq\cF(C,P)$, proving that $\Pi_P\cF$ is increasing.
If $i,j\in P$, $i<j$, $j\in A$ and $i\notin A$, by the left-shiftedness of $\cF$,
$$
 \cF(A,P)\subseteq
 \cF((A\setminus\{j\})\cup\{i\},P).
$$
Thus the majority projection is left-shifted as well; the right-shifted
case is the same with the order reversed.

For $A,C\in\Pi_P\cF$, both $\cF(A,P)$ and $\cF(C,P)^c$ have
measure greater than $1/2$ in $2^{P^c}$.
Their intersection is therefore nonempty, so there is
$B\subseteq P^c$ with $A\cup B\in\cF$ and
$C\cup(P^c\setminus B)\in\cF$.
It follows that
$|A\cap C|=|(A\cup B)\cap(C\cup(P^c\setminus B))|\ge t$.

Finally, $\min\{s,1-s\}\le2s(1-s)$ for $0\le s\le1$.
Use \eqref{eq:majority-best-approximation} and apply the Poincar\'e
inequality \eqref{eq:classical-poincare} to each restricted projection:
\[
\mu^n(\cF\triangle\Pi_P^*\cF)\le2\E_{A\subseteq P}\operatorname{Var}(\mathbbm{1}_{\cF(A,P)})\le\frac12\E_{A\subseteq P}I[\cF(A,P)]
 =\frac12 I_{P^c}[\cF]. \qedhere
\]
\end{proof}

\subsection{Proof overview}

Note that the product construction implies $w(n,p,q)\geq B_n$.
By Fact \ref{fact-reduction}, to prove Theorem \ref{thm:main} it suffices to show $w(n,p,q)\leq B_n$ when $n-p-q$ is odd. Suppose for contradiction that
 there is a $(p,q)$-family $\cF$ with $\mu(\cF)>B_n$.
By Lemma~\ref{lem:closure}, there are  increasing families
$\cA,\cB\subseteq2^{[n]}$ such that $\cA$ is $p$-intersecting and
left-shifted, $\cB$ is $q$-intersecting and right-shifted, and
$\cF\subset \cA\cap\cB^c$.
For $f=\mathbbm{1}_{\cA}$ and $g=\mathbbm{1}_{\cB}$, we have
$$
\mu(\cA\cap\cB^c)=\E[f(1-g^*)]=\mu(\cA)\,\mu(\cB)-\Cov(f,g^*),
 \qquad g^*(x)=1-g(1-x).
 \eqno(\ref{eq:pair})
$$
Then the assumption $\mu(\cF) >B_n$ implies
\[
\Cov(f,g^*)<\mu(\cA)\,\mu(\cB)-\mu(\cA\cap\cB^c) \leq \mu(\cA)\,\mu(\cB)-\mu(\cF) <\mu(\cA)\,\mu(\cB)-B_n.
\]

We first show that the two families $\cA,\cB$ are close to Katona families on disjoint subsets of $[n]$. For every increasing $t$-intersecting family $\cF\subseteq2^{[n]}$, where $t\ge2$, and every $t\le m\le n$ with $m\equiv t\pmod2$, we prove
$$
I[\cF]-I[\cK(n,m,t)]
\ge\frac{2(m+t(t-2))}{t-1}
\bigl(\mu(\cF)-\mu(\cK(n,m,t))\bigr).
\eqno(\ref{eq:profile-supporting})
$$
Lemma~\ref{lem:tails} also bounds $I_{P^c}[\cF]$, the influence
outside an arbitrary coordinate set $P$, in terms of $|P|$ and
$\mu(\cF)$. Applying these bounds to $f$ and $g$ and summing by
parts, one can derive an upper bound on  $\Cov(f,g^*)$. A Talagrand-type correlation inequality (see Lemma~\ref{lem:correlation}) provides a lower bound on $\Cov(f,g^*)$. Combining the two bounds, we arrive at
$$
 \frac{\mathcal{I}[f,g]}{4\bigl(2+\ln(n/\mathcal{I}[f,g])\bigr)}
 \le\Cov(f,g^*)
 <\frac{10(p+q-2)}{11\sqrt n}\,\mathcal{I}[f,g],
 \qquad \mathcal{I}[f,g]=\sum_{i=1}^n I_i[f]I_i[g].
$$
Thus by the Harris--Kleitman inequality we obtain that
$$
 0<\mathcal{I}[f,g]
 <n\exp\left(2-\frac{11\sqrt n}{40(p+q-2)}\right).
 \eqno(\ref{eq:small-overlap})
$$

Set $P=\{i:I_i[f]\ge I_i[g]\}$ and take the majority projection
$$
\Pi_P\cA=\left\{S\subseteq P:\mu^{n-|P|}(\cA(S,P))>\frac{1}{2}\right\};
$$
project $\cB$ onto $P^c$ in the same way.
Lemma~\ref{lem:stability} approximates these projections by
Katona families. Proposition~\ref{prop:localization} yields a partition $[n]=X\sqcup\{z\}\sqcup Y$, with $X=[n_1]$, $z=n_1+1$, $Y=[n_1+2,n]$, $n_1\equiv p\pmod2$ and $n_2=|Y|\equiv q\pmod2$, such that
$$
\mu(\cA\triangle\cK(n,X,p))\le\frac1{64n},\qquad
\mu(\cB\triangle\cK(n,Y,q))\le\frac1{64n}.
$$

For $y\subseteq Y\cup\{z\}$ and $x\subseteq X\cup\{z\}$, write
$\cA_y=\cA(y,Y\cup\{z\})\subseteq2^X$ and $\cB_x=\cB(x,X\cup\{z\})\subseteq2^Y$. Define
$$
 \begin{aligned}
 \phi_\cA(y)&=\bigl(\mu^{n_1}(\cA_y)
       -\mu(\cK(n,X\setminus\{1\},p-1))\bigr)_+,\\
 \phi_\cB(x)&=\bigl(\mu^{n_2}(\cB_x)
       -\mu(\cK(n,Y\setminus\{n\},q-1))\bigr)_+.
 \end{aligned}
$$
The subtracted measures are those of the Katona projections obtained
by fixing coordinate $1$ or $n$, respectively, to one; see
\eqref{eq:threshold-one-projection}.
The stronger approximation in Proposition~\ref{prop:localization}
gives
$$
 \mu^{n_2+1}(\phi_\cA>0),\ \mu^{n_1+1}(\phi_\cB>0)\le e^{-L},
 \qquad L=\frac{n}{96(p+q-2)^2}.
$$
Let $x_+=\max\{x,0\}$ for $x\in\mathbb R$.
In Sections~\ref{sec:exterior} and~\ref{subsec:rare-expansion}, we prove that
$$
 \bigl(\mu(\cA)-w(|X|,p)\bigr)_+
 \le\frac{16(p-1)\sqrt{n_2+1}}{n_1\sqrt L}\,\E\phi_\cA,
 \eqno(\ref{eq:final-density-gain})
$$
and the corresponding bound for  $\bigl(\mu(\cB)-w(|Y|,q)\bigr)_+$.

The main combinatorial step compares added and missing sets in
complementary projections. After the compression of
Lemma~\ref{lem:simplicial-compression}, each projection consists of
the last subsets in simplicial order and has the same measure as
before. Let $y\subseteq y'\subseteq Y\cup\{z\}$ with
$|y'\setminus y|\le i$, where $0\le i\le p-1$.
If $A\in\cA_{y'}$ and $B\subseteq A$ with
$|A\setminus B|=p-1-i$, then
$(X\setminus B)\cup y^c\notin\cA$: its intersection with
$A\cup y'\in\cA$ has size at most $p-1$.
Here $y^c=(Y\cup\{z\})\setminus y$.
We count the excluded sets separately at each level.
For the first level below the Katona threshold, double counting
proves \eqref{eq:first-layer}. For the other levels, the same
intersection calculation gives \eqref{eq:exclusion}.
Harper's theorem supplies the neighborhood expansion
\eqref{eq:radius-expansion}, and
Lemma~\ref{lem:interval-coupling} combines the resulting estimates
without counting a missing set more than once.
The latter lemma is proved in Appendix~\ref{app:shadow-tools}
using \cite[Corollary~2.1]{BGI}.

For the covariance estimate, the approximation errors at most
$1/(64n)$ permit the constant $24$ in
\eqref{eq:smoothed-correlation}. Apply this bound and edge
isoperimetry to the indicators of $\{\phi_\cA>s\}$ and
$\{\phi_\cB>s\}$, and integrate over $s\ge0$.
By conditional Harris--Kleitman, we obtain that
$$
 \Cov(f,g^*)\ge\frac1{24}\left(
       \sqrt{\frac{L-\ln2}{n_2}}\,\E\phi_\cA+
       \sqrt{\frac{L-\ln2}{n_1}}\,\E\phi_\cB\right).
 \eqno(\ref{eq:rare-correlation})
$$
Since  $\mu(\cA),\mu(\cB),w(|X|,p),w(|Y|,q)\leq 1/2$, we have
$$
\begin{aligned}
&\mu(\cA)\,\mu(\cB)-w(|X|,p)\,w(|Y|,q)\le\frac12\bigl(\mu(\cA)-w(|X|,p)\bigr)_++\frac12\bigl(\mu(\cB)-w(|Y|,q)\bigr)_+.
\end{aligned}
$$
The coefficient comparison in \eqref{eq:final-comparison}, using
\eqref{eq:block-ratio} and $L>256(p+q-2)$, bounds this expression
by $\frac23\Cov(f,g^*)$. Substitution into \eqref{eq:pair} yields
$\mu(\cA\cap\cB^c)\le
w(|X|,p)\,w(|Y|,q)\le B_n$, a contradiction.

\subsection{Binomial and numerical estimates}\label{subsec:numerical}

The following bounds will be used to keep all constants explicit. Their proofs are given in Appendix~\ref{app:auxiliary}.

\begin{lemma}\label{lem:binomial}
For $t\ge2$ and any admissible $m$ for $t$,
\begin{equation}\label{eq:explicitbinomial}
 \theta(m,t)\le\frac{\sqrt2}{\sqrt m},\qquad
 \frac1{4\sqrt m}\le\frac12-w(m,t)\le\frac{t-1}{\sqrt m}.
\end{equation}
We also have
\begin{align}
 m\ge8t^2&\mbox{ implies }
 \theta(m,t)\ge\frac1{2\sqrt m},\label{eq:centralw}\\
 m\ge128&\mbox{ implies }
 \theta(m,t)\le\frac4{5\sqrt m},\label{eq:profile-upper-sharp}\\
 m\ge256t^2&\mbox{ implies }
 \theta(m,t)\ge\frac{27}{34\sqrt m}.\label{eq:profile-lower-central}
\end{align}
Finally,
\begin{equation}\label{eq:closer-half-dimension}
 \frac12-w(m,t)<\frac1{96}
\mbox{ implies } m>256t^2+2.
\end{equation}
\end{lemma}

\begin{lemma}\label{lem:numerical}
Suppose $p,q\ge2$ and $n\ge2^{15}(p+q-2)^3$.
Then $n\ge2^{16}(p+q-2)^2$ and
\begin{align}
 \frac{11\sqrt n}{80(p+q-2)}
 &>\frac52\ln(n+1)+\ln(4 e),\label{eq:initiallog}\\
 \frac{n}{18^2(p+q-2)^2}
 &\ge16\ln(n+1),\label{eq:weak-bootstrap-number}\\
 \frac{n}{96(p+q-2)^2}&\ge \max \left\{32\ln(n+1), 256(p+q-2),192^2(p+q-2)^2/n\right\}.\label{eq:finalnumbers}
\end{align}
If $n/[2(p+q-2)^{2/3}]\le m\le n$, then $m\ge8192$,
$\frac{n}{96(p+q-2)^2}\le m/64$, and $\frac{nm}{96(p+q-2)^2}\ge1024(p+q-2)^2$.
For $h\in\{p-1,q-1\}$ and $0<u\le n$,
\begin{equation}\label{eq:coarse-error-number}
 \left(\frac{3nu}{h}
       e^{-11\sqrt n/[80(p+q-2)]}\right)^{4/5}
 \le\frac3{16\sqrt u},
\end{equation}
provided $u\le5n/8$ when $p+q-2=2$.
\end{lemma}

\begin{lemma}
For $p,q\ge2$ and $n\ge2^{16}(p+q-2)^2$,
\begin{align}\label{ineq-BnlowerB}
 B_n
 \ge\frac14-\frac{p+q-2}{\sqrt n} \geq  \frac14-\frac1{256}.
 \end{align}
\end{lemma}
\begin{proof}
Choose $n_1,n_2$ admissible for $p,q$, respectively, such that
$$
 \frac n2-2\le n_1,n_2\le\frac n2.
$$
Then $n_1+n_2\le n$ and $n_1,n_2\ge n/4$.
By \eqref{eq:explicitbinomial},
\begin{align*}
 B_n
 \ge w(n_1,p)w(n_2,q)
 &\ge
 \left(\frac12-\frac{2(p-1)}{\sqrt n}\right)
 \left(\frac12-\frac{2(q-1)}{\sqrt n}\right)\ge\frac14-\frac{p+q-2}{\sqrt n}.
 \end{align*}
Since $n\ge2^{16}(p+q-2)^2$, \eqref{ineq-BnlowerB} follows.
 \end{proof}

 Let $\eps\in\{0,1\}$ satisfy $\eps\equiv n-p-q\pmod2$.

\begin{claim}\label{cl:product-increment}
Suppose $p,q\ge2$ and $n\ge2^{16}(p+q-2)^2$.
Let $n_1$ and $n_2$ be admissible for $p$ and $q$, respectively, with
$256p^2\le n_1\le n$ and $256q^2\le n_2\le n$.
If $n_1+n_2\ge n-\eps$, then
\begin{align}\label{ineq-3.1}
 w(n_1,p)w(n_2,q)-B_n
 \le
 \frac{(p+q-2)(n_1+n_2-n+\eps)}
      {5(n-1)\sqrt{2(n+1)}}.
\end{align}
\end{claim}

The proof is given in Appendix~\ref{app:claims}.

\begin{lemma}
\label{cl:overlap-endpoint}
Suppose $p,q\ge2$ and $n\ge2^{16}(p+q-2)^2$.
Let $n_1$ and $n_2$ be admissible for $p$ and $q$, respectively, with
$256p^2\le n_1\le n$ and $256q^2\le n_2\le n$.
Then
\begin{equation}\label{eq:overlap-endpoint}
\frac14\theta(n_1,p)\theta(n_2,q)(n_1+n_2-n)_+
 \ge\frac{11\sqrt n}{10(p+q-2)}
       \bigl(w(n_1,p)w(n_2,q)-B_n\bigr).
\end{equation}
\end{lemma}

The proof is given in Appendix~\ref{app:claims}.

\begin{lemma}
Suppose $p,q\ge2$ and $n\ge2^{16}(p+q-2)^2$.
Let $n_1$ and $n_2$ be admissible for $p$ and $q$, respectively, with
$256p^2\le n_1\le n$ and $256q^2\le n_2\le n$.
Suppose also that $n_1+n_2=n-\eps$.
If
\begin{equation}\label{ineq:ratio-bound}
\left(\frac{p-1}{2(q-1)}\right)^{2/3}
 \le\frac {n_1}{n_2}\le\left(\frac{2(p-1)}{q-1}\right)^{2/3},
 \qquad n_1,n_2\ge\frac{n}{2(p+q-2)^{2/3}},
\end{equation}
then
\begin{equation}\label{eq:geometric}
 \frac{(p-1)\sqrt{n(n_2+2)}}{n_1}\le\frac98(p+q-2),
 \qquad
 \frac{(p-1)\sqrt{n_2(n_2+1)}}{n_1}\le\frac78(p+q-2),
\end{equation}
and the two inequalities obtained by exchanging $(n_1,p)$ and $(n_2,q)$ hold as well.
\end{lemma}

\begin{proof}
By \eqref{ineq:ratio-bound} and weighted AM--GM we obtain that
$$
 \begin{aligned}
 \frac{(p-1)\sqrt{n(n_2+2)}}{n_1}
 &\le\frac{33}{32}
 \sqrt{2^{4/3}(p-1)^{2/3}(q-1)^{4/3}
       +2^{2/3}(p-1)^{4/3}(q-1)^{2/3}}\\
 &\le\frac{33}{32}\sqrt{\frac76}\,(p+q-2)<\frac98(p+q-2),\\
 \frac{(p-1)\sqrt{n_2(n_2+1)}}{n_1}
 &\le\frac{49}{48}\,2^{2/3}(p-1)^{1/3}(q-1)^{2/3}
 \le\frac{49}{48}\cdot\frac67(p+q-2).
 \end{aligned}
$$
Here $n/(n_1+n_2),(n_2+2)/n_2\le33/32$ and
$\sqrt{(n_2+1)/n_2}\le49/48$.
\end{proof}

\section{Stability results}\label{sec:stability}

\subsection{Stability of Katona's theorem}\label{subsec:katona-stability}

Let us first prove a stability result of the Katona Theorem.
\Needspace{8\baselineskip}
\begin{lemma}\label{lem:stability}
Let $n\geq t\ge2$ be integers and let $\cF\subset 2^{[n]}$ be $t$-intersecting.
\begin{itemize}
  \item[(i)] If $n-t$  is even, then
  \begin{align}\label{eq:stability-e}
\mu(\cF\triangle\cK(n,[n],t))\le\frac{n-1}{t-1}\bigl(w(n,t)-\mu(\cF)\bigr).
  \end{align}
  \item[(ii)] If $n-t$ is odd and $\cF$ is left-shifted, then
   \begin{align}\label{eq:stability-o}
\mu(\cF\triangle\cK(n,[n-1],t))\le\frac{n-2}{t-1}\bigl(w(n-1,t)-\mu(\cF)\bigr).
  \end{align}
\end{itemize}
\end{lemma}
\begin{proof}
For $n=t$ the assertion is immediate. If $n=t+1$, a left-shifted
$t$-intersecting family contains no $t$-sets containing $n$: shifting $n$ to
the unique missing element of $[n-1]$ would create two members intersecting
in $t-1$ elements. Thus it is contained in $\cK(n,[n-1],t)$, and the
assertion is again immediate. Thus we may  assume $n-t\ge 2$.

By \eqref{eq:classical-shadow}, every $t$-intersecting
$\cF^{(j)}\subseteq\binom{[n]}j$ with $j\geq t$ satisfies
\begin{equation}\label{eq:katona}
 |\partial^{t-1}\cF^{(j)}|
 \ge\frac{\binom{2j-t}{j-t+1}}{\binom{2j-t}j}|\cF^{(j)}|
 =\frac{j}{j-t+1}|\cF^{(j)}|.
\end{equation}

First suppose $n-t=2r$. Then
 $$
 \cF\setminus \cK(n,[n],t) = \bigcup_{t\leq j\leq t+r-1} \cF^{(j)}.
 $$
For each $G\in\partial^{t-1}\cF^{(j)}$, there is some $H\in\cF^{(j)}$ with $G\subset H$ and $|H\cap([n]\setminus G)|=t-1$, implying $[n]\setminus G \notin \cF$. Therefore,
$$
(\partial^{t-1}\cF^{(j)})^c \subset \binom{[n]}{n-j+t-1}\setminus \cF, j=t,t+1,\ldots,t+r-1.
$$
Thus for $j=t,t+1,\ldots,t+r-1$,
\begin{align}\label{ineq-new3.2}
\left|\binom{[n]}{n-j+t-1}\setminus \cF\right| \geq |(\partial^{t-1}\cF^{(j)})^c| =|\partial^{t-1}\cF^{(j)}|\overset{\eqref{eq:katona}}{\geq} \frac{j}{j-t+1}|\cF^{(j)}|\geq \frac{t+r-1}{r}|\cF^{(j)}|.
\end{align}
Note that $n=t+2r$ and $j\leq t+r-1$ imply $n-j+t-1\geq t+r$. It follows that
$$
(\partial^{t-1}\cF^{(j)})^c \subset \cK(n,[n],t)\setminus \cF.
$$
Summing over $j=t,t+1,\ldots,t+r-1$, we obtain that
$$
|\cK(n,[n],t)\setminus \cF| \geq  \frac{t+r-1}{r} \sum_{t\leq j\leq t+r-1}|\cF^{(j)}| = \frac{t+r-1}{r} | \cF\setminus \cK(n,[n],t)|.
$$
Using $|\cF\triangle\cK(n,[n],t)|=|\cK(n,[n],t)\setminus \cF|+| \cF\setminus \cK(n,[n],t)|$ and $|\cK(n,[n],t)|-|\cF|= |\cK(n,[n],t)\setminus \cF|-| \cF\setminus \cK(n,[n],t)|$, we conclude that
$$
 (t-1)|\cF\triangle\cK(n,[n],t)|
 \le(n-1)\bigl(|\cK(n,[n],t)|-|\cF|\bigr),
$$
which proves \eqref{eq:stability-e}.

Next assume $n=t+2r+1$ and let $\cF\subseteq2^{[n]}$ be left-shifted. By \eqref{ineq-new3.2}, we have
for $j=t,t+1,\ldots,t+r-1$,
\begin{align}\label{ineq-2.2}
\left|\binom{[n]}{n-j+t-1}\setminus \cF\right| \geq \frac{t+r-1}{r}|\cF^{(j)}|.
\end{align}
Since $n=t+2r+1$ and $j\leq t+r-1$ imply $n-j+t-1\geq t+r+1$,
$$
(\partial^{t-1}\cF^{(j)})^c \subset \cK(n,[n-1],t)\setminus \cF.
$$
Let us consider the $(t+r)$-sets in $\cF$. Define
$$
\cH= \{H\in \cF^{(t+r)}\colon n\in H\},\ \cG= \binom{[n-1]}{t+r}\setminus \{H\in \cF^{(t+r)}\colon n\notin H\}.
$$
Note that
$$
\cG= (\cK(n,[n-1],t)\setminus \cF)\cap \binom{[n]}{t+r}\mbox{ and }\cH= (\cF\setminus  \cK(n,[n-1],t))\cap \binom{[n]}{t+r}.
$$
\begin{claim}
 \begin{align}\label{ineq-2.3}
|\cG| \geq \frac{t+r-1}{r}|\cH|.
\end{align}
\end{claim}
\begin{proof}
Let $\cH' =\{H\setminus \{n\}\colon H\in \cH\}$. Since $\cF$ is $t$-intersecting, we infer that
$$
\left\{[n-1]\setminus E\colon E\in \partial^{(t-1)} \cH'\right\} \subset \cG.
$$
Since $\cF$ is left-shifted and $n=t+2r+1$, we infer that $\cH'$ is also $t$-intersecting. By \eqref{eq:katona}, we obtain that
\begin{equation*}
|\cG| \geq |\partial^{(t-1)} \cH'| \geq \frac{t+r-1}{r} |\cH'|=\frac{t+r-1}{r} |\cH|. \qedhere
\end{equation*}
\end{proof}

By \eqref{ineq-2.2} and \eqref{ineq-2.3}, we infer that
$$
|\cK(n,[n-1],t)\setminus \cF| \geq \frac{t+r-1}{r} | \cF\setminus \cK(n,[n-1],t)|.
$$
Using $|\cF\triangle\cK(n,[n-1],t)|=|\cK(n,[n-1],t)\setminus \cF|+| \cF\setminus \cK(n,[n-1],t)|$ and $|\cK(n,[n-1],t)|-|\cF|= |\cK(n,[n-1],t)\setminus \cF|-| \cF\setminus \cK(n,[n-1],t)|$, we conclude that
$$
 (t-1)|\cF\triangle\cK(n,[n-1],t)|
 \le(n-2)\bigl(|\cK(n,[n-1],t)|-|\cF|\bigr),
$$
which proves \eqref{eq:stability-o}.
\end{proof}

\subsection{Influence stability}\label{sec:influence}

Let us recall the {\it generating-set method}, developed by Ahlswede and Khachatrian~\cite{AK}.
We use the following formulation for increasing families.

For an increasing family $\cF$, let $\cG(\cF)$ consist of its
inclusion-minimal members, which is often called the {\it generating set} of $\cF$. Then
$\cF=\langle\cG(\cF)\rangle$, where
$\langle\cG\rangle=\{A:G\subseteq A\text{ for some }G\in\cG\}$.
The \emph{support} of $\cF$ is $\bigcup_{G\in\cG(\cF)}G$.
For a nonempty left-shifted family $\cF$, its support is an initial interval $[\ell]$.
Define the {\it boundary generating set} as
$\cG^*=\{G\in\cG:\ell\in G\}$. Recall that  $(\cG^*)^{(i)}=\{G\in\cG^*:|G|=i\}$.

\begin{lemma}[Generating-set properties~\cite{AK}]\label{lem:generators}
Let $\cF\subseteq2^{[n]}$ be increasing, left-shifted and $t$-intersecting. Let $[\ell]$ be the support of $\cF$.
Then
\begin{itemize}
  \item[(i)] If $G\in \cG^*$ and $G\setminus \{\ell\}\subsetneq F\subseteq[\ell]$, then $F\in \cF$.
  \item[(ii)] If $G_1,G_2\in\cG^*$ satisfy $|G_1\cap G_2|=t$, then $G_1\cup G_2=[\ell]$ and $|G_1|+|G_2|=\ell+t$.
\end{itemize}
\end{lemma}

\begin{proof}
For (i), if $\ell\in F$, then $G\subseteq F$ and $F\in \langle\cG\rangle=\cF$. Otherwise choose
$j\in F\setminus G$. By left-shiftedness,  $(G\setminus\{\ell\})\cup\{j\}\in\cF$,
and by increasingness  $F\in\cF$ follows.
For (ii), if $j\in[\ell]\setminus(G_1\cup G_2)$, by replacing $\ell$ by $j$
in $G_1$ we obtain a member of $\cF$ whose intersection with $G_2$ has size
$t-1$, a contradiction.
\end{proof}

By using the generating set method, we prove the next lemma, which shows that if a  $t$-intersecting increasing family is close to a Katona family in measure, then its total influence is close to the total influence of the  Katona family as well.

\begin{lemma}\label{lem:profile}
Let $n\geq t\geq2$ and let $\cF\subseteq2^{[n]}$ be increasing
and $t$-intersecting. For every admissible $m\ge t$,
\begin{equation}\label{eq:profile-supporting}
 I[\cF]\ge J(m,t)+\frac{2(m+t(t-2))}{t-1}
                  \bigl(\mu(\cF)-w(m,t)\bigr).
\end{equation}
In particular, if
$\mu(\cF)=\alpha w(m-2,t)+(1-\alpha) w(m,t)$
for some $\alpha\in[0,1]$ and admissible $m\ge t$, then
\begin{equation}\label{eq:profile-interpolation}
 I[\cF]\ge \alpha J(m-2,t)+(1-\alpha) J(m,t).
\end{equation}
\end{lemma}

\begin{proof}
Direct computation shows that, for any admissible $m\ge t$,
\begin{align}
 &w(m,t)-w(m-2,t)
=\frac{(t-1)\theta(m,t)}{2(m-1)},
 \label{eq:incrementa}\\[3pt]
 &J(m,t)-J(m-2,t)
=\frac{m+t(t-2)}{m-1}\theta(m,t).
 \label{eq:incrementJ}
\end{align}
Fix $\lambda>0$.
For $\cF\subseteq2^{[n]}$, define
$$
 f_\lambda(\cF):=\mu(\cF)-\lambda I[\cF].
$$
Among all increasing $t$-intersecting families in $2^{[n]}$,
choose a maximizer $\cF$ of $f_\lambda$ with minimum support size.

Since $\cF$ is increasing, by \eqref{ineq-2.1} and double counting we infer that
$$
 I[\cF]
 =\sum_{i=1}^n
       \frac{|\cF(i)|-|\cF(\bar i)|}{2^{n-1}}
 =2^{-n+1}\sum_{A\in\cF}(2|A|-n).
$$
Consequently,
\begin{equation}\label{eq:profile-functional}
 f_\lambda(\cF)
 =2^{-n}\sum_{A\in\cF}
       \bigl(1-2\lambda(2|A|-n)\bigr).
\end{equation}
The shifting operators $S_{ij}$, $1\le i<j\le n$, preserve
increasingness, $t$-intersection and the number of members of
each size.
By \eqref{eq:profile-functional}, they also preserve
$f_\lambda$. Moreover, the shifting operators do not increase the size of its support.
Thus a maximizer $\cF$ may be assumed left-shifted. Let $\cG=\cG(\cF)$ and let $\ell$ be its support size. Since $\cG$ is $t$-intersecting, we infer that  $\ell\ge t$ if $\cF\neq \emptyset$.

\begin{claim}\label{cl:profile-support}
\begin{equation}\label{eq:profile-support-size}
 \ell<\frac{t-1}{2\lambda}-t(t-2).
\end{equation}
Moreover,
\begin{equation}\label{eq:generator-middle}
\ell \equiv t\pmod2 \mbox{ and } \cG^*\subset \binom{[\ell]}{(\ell+t)/2}
\end{equation}
\end{claim}

Let us first show that Claim \ref{cl:profile-support} implies the lemma. We  prove \eqref{eq:profile-supporting} by induction on
admissible $m$, with $\lambda=\frac{t-1}{2(m+t(t-2))}$.
For $m=t$, we have $\lambda=1/(2t)$.
A  maximizer $\cF$ would have $\ell<t$ by
\eqref{eq:profile-support-size}. Since $\cG(\cF)$ is $t$-intersecting, we infer that $\cF=\varnothing$.
Hence the maximum of $f_\lambda$ is zero, and
$$
 I[\cF]\ge 2t\mu(\cF)
$$
for every increasing $t$-intersecting family $\cF$.
Since $J(t,t)=2t w(t,t)$, this is
\eqref{eq:profile-supporting} for $m=t$.

Let $m>t$ and assume that \eqref{eq:profile-supporting}
has been proved for $m-2$ and for every increasing
$t$-intersecting family.
Choose an increasing $t$-intersecting family $\cF$ with $f_{\lambda}(\cF)$ maximal and the support size minimal as above.
If $\cF\ne\varnothing$, then
\eqref{eq:profile-support-size} implies $\ell<m$.
By \eqref{eq:generator-middle}, $\ell\equiv t\pmod2$,
so $\ell\le m-2$.
Theorem~\ref{thm-katona} and \eqref{eq:incrementa} imply
$$
 \mu(\cF)\le w(\ell,t)\le w(m-2,t).
$$
The same upper bound holds if $\cF=\varnothing$.

By the induction hypothesis,
\begin{align*}
I[\cF]\ge J(m-2,t)+\frac{2(m-2+t(t-2))}{t-1}
                  \bigl(\mu(\cF)-w(m-2,t)\bigr).
\end{align*}
It follows that
$$
 \begin{aligned}
 f_\lambda(\cF) &= \mu(\cF) -\lambda I[\cF]\\[3pt]
 &\le w(m-2,t)-\lambda J(m-2,t)+\frac{2}{m+t(t-2)}
       \bigl(\mu(\cF_0)-w(m-2,t)\bigr)\\
 &\le w(m-2,t)-\lambda J(m-2,t)\\
 &=w(m,t)-\lambda J(m,t),
 \end{aligned}
$$
where the last equality follows from
\eqref{eq:incrementa} and \eqref{eq:incrementJ}.
By maximality, the same upper bound holds for
$f_\lambda(\cF)$ for every increasing
$t$-intersecting $\cF\subseteq2^{[n]}$.
Rearranging proves \eqref{eq:profile-supporting} and
completes the induction.

Finally, \eqref{eq:incrementa} and \eqref{eq:incrementJ} show that
the right-hand side of \eqref{eq:profile-supporting} equals
$\alpha J(m-2,t)+(1-\alpha) J(m,t)$ when
$\mu(\cF)=\alpha w(m-2,t)+(1-\alpha) w(m,t)$.
This proves \eqref{eq:profile-interpolation}.
\end{proof}

\begin{proof}[Proof of Claim~\ref{cl:profile-support}]
Note that  $\cF =\cF_{\mid [\ell]} \times 2^{[\ell+1,n]}$ implies $\mu(\cF)= \mu(\cF_{\mid [\ell]})$ and $I[\cF]= I[\cF_{\mid [\ell]}]$. Thus we may assume $n=\ell$.
For $t\le i,j\le \ell$ with $i\ne j$ and $i+j=\ell+t$, define
$$
 \cF_{i,j}
 =(\cF\setminus(\cG^*)^{(i)})
   \cup\{G\setminus\{\ell\}:G\in(\cG^*)^{(j)}\},
$$
and define $\cF_{j,i}$ symmetrically.
Since sets in $(\cG^*)^{(i)}$ are all  inclusion-minimal members of $\cF$, deleting them preserves the increasingness.  By Lemma \ref{lem:generators} (i), every proper superset of $G\setminus\{\ell\}$ is in $\cF$.
Hence $\cF_{i,j}$ is increasing. To check $t$-intersection, it suffices to check the retained generators
and the added sets $G\setminus\{\ell\}$, $G\in(\cG^*)^{(j)}$.
Any violation would require two original boundary generators intersect
in exactly $t$ elements. By Lemma~\ref{lem:generators}(ii), their sizes
sum to $\ell+t$: a retained generator would therefore have size $i$,
and two added generators would require $i=j$. Both are excluded.
Thus $\cF_{i,j}$ is $t$-intersecting. The same holds for $\cF_{j,i}$.

By \eqref{eq:profile-functional},
$$
 f_\lambda(\cF_{i,j}) -f_\lambda(\cF)
 =2^{-\ell}\left(-\sum_{A\in(\cG^*)^{(i)}}
       \bigl(1-2\lambda(2|A|-\ell)\bigr) + \sum_{A\in(\cG^*)^{(j)}}
       \bigl(1-2\lambda(2|A|-2-\ell)\bigr)\right).
$$
It follows that
$$
 f_\lambda(\cF_{i,j})+f_\lambda(\cF_{j,i})
       -2f_\lambda(\cF)
 =2^{-\ell}4\lambda
       \bigl(|(\cG^*)^{(i)}|+|(\cG^*)^{(j)}|\bigr).
$$
The maximality of $f_\lambda(\cF)$ implies
$(\cG^*)^{(i)}=(\cG^*)^{(j)}=\varnothing$ whenever
$i\ne j$ and $i+j=\ell+t$.
Since $\mathcal G^*\ne\varnothing$, we conclude that
\eqref{eq:generator-middle} holds.

For each $i<\ell$, define
$$
 \cF_i
 =\bigl(\cF\setminus\{G\in\cG^*:i\in G\}\bigr)
   \cup
   \{G\setminus\{\ell\}:G\in\cG^*,\ i\notin G\}.
$$
By Lemma \ref{lem:generators} (i), every proper superset of $G\setminus\{\ell\}$ is in $\cF$.
Hence $\cF_i$ is increasing. It is generated by the members of $\cG\setminus\cG^*$ together with
$G\setminus\{\ell\}$ for $G\in\cG^*$ not containing $i$.
A violation of $t$-intersection would require two original generators
containing $\ell$ and intersect in exactly $t$ elements. Their union would
be $[\ell]$ by Lemma~\ref{lem:generators}(ii), impossible since both
corresponding boundary generators omit $i$.
Thus $\cF_i$ is $t$-intersecting.

Note that  the pairs
$(A,A\cup\{\ell\})$ with $A\notin\cF$ and
$A\cup\{\ell\}\in\cF$ are precisely
$$
 (G\setminus\{\ell\},G),
 \qquad G\in\mathcal G^*.
$$
In $\cF_i$, both members of such a pair are absent if
$i\in G$, and both are present if $i\notin G$.
All other pairs are unchanged.
Consequently $\cF_i$ has support size smaller than $\ell$.
The minimality of the support implies
$$
 f_\lambda(\cF_i)<f_\lambda(\cF)
 \qquad\text{for every }i<\ell.
$$

Since $\cG^*\subset \binom{[\ell]}{(\ell+t)/2}$,
using  \eqref{eq:profile-functional} we obtain that
\begin{align*}
f_\lambda(\cF_i)-f_\lambda(\cF)&=2^{-\ell}\left(-\sum_{G\in\cG^*:i\in G}
       \bigl(1-2\lambda(2|G|-\ell)\bigr)+
       \sum_{G\in\cG^*:i\notin G}
       \bigl(1-2\lambda(2(|G|-1)-\ell)\bigr)\right)\\[3pt]
       &=2^{-\ell}\left(-\sum_{G\in\cG^*:i\in G}
       \bigl(1-2\lambda t\bigr)+
       \sum_{G\in\cG^*:i\notin G}
       \bigl(1-2\lambda(t-2)\bigr)\right).
\end{align*}
Thus,
$$
 \begin{aligned}
 \sum_{i=1}^{\ell-1}
       \bigl(f_\lambda(\cF_i)-f_\lambda(\cF)\bigr)&
      =
 2^{-\ell}|\mathcal G^*|
 \left[
   \frac{\ell-t}{2}\bigl(1-2\lambda(t-2)\bigr)
   -\left(\frac{\ell+t}{2}-1\right)(1-2\lambda t)
 \right]\\
 &=
 2^{-\ell}|\mathcal G^*|
       \bigl(1-t+2\lambda[\ell+t(t-2)]\bigr)<0.
 \end{aligned}
$$
Since $\mathcal G^*\ne\varnothing$, \eqref{eq:profile-support-size} follows.
\end{proof}

\begin{lemma}\label{lem:partial-influence}
Let $f\colon \{0,1\}^n\mapsto \{0,1\}$ be an increasing function. Then for any $P\subset [n]$,
\begin{align}\label{ineq-IPf}
I_P[f]\leq \sqrt{|P|}.
\end{align}
\end{lemma}
This follows by fixing the coordinates outside $P$, applying the usual
total-influence bound to the remaining coordinates, and averaging;
see~\cite[Theorem~2.33 and Exercise~2.23]{OD}. We include its short proof.
\begin{proof}
Since $f$ is increasing and Boolean,
\begin{align}\label{ineq-Iif}
 I_i[f]
 =\E[f\mid X_i=1]-\E[f\mid X_i=0]
 =4\E[(X_i-\tfrac12)f].
\end{align}
Let
$Z=\sum_{i\in P}X_i-|P|/2$.
Summing over $i\in P$ we obtain that $I_P[f]=4\E[Zf]$.
Moreover, $f\in \{0,1\}$ implies $Zf\le Z_+$, and
$\E Z=0$ implies $\E Z_+=\tfrac12\E|Z|$.
The variables $X_i$ are independent Bernoulli random variables
with parameter $1/2$, so $\E Z^2=|P|/4$.
Hence, by the Cauchy--Schwarz inequality,
\begin{equation*}
 I_P[f] =4\E[Zf] \le4\E Z_+ =2\E|Z|\le2\sqrt{\E Z^2} =\sqrt{|P|}. \qedhere
\end{equation*}
\end{proof}

\begin{lemma}\label{lem:tails}
Let $t\ge2$, and suppose that $f$ is the indicator function of an
increasing $t$-intersecting subfamily of $2^{[n]}$.
Assume that, for some $\alpha\in[0,1]$ and  admissible $m$ with $n\ge m\ge16t^2$,
$$
 \E f=\alpha w(m,t)+(1-\alpha)w(m+2,t).
$$
Then for every $P\subseteq[n]$,
\begin{equation}\label{eq:tails}
 I_{P^c}[f]\ge
 \frac12\Bigl[
   \alpha\theta(m,t)(m-|P|)_+
   +(1-\alpha)\theta(m+2,t)(m+2-|P|)_+
 \Bigr].
\end{equation}
\end{lemma}

\begin{proof}
Applying \eqref{eq:profile-interpolation} with $m$ replaced by $m+2$, we obtain that
\begin{equation}\label{eq:tails-two-point-profile}
 I[f]\ge \alpha J(m,t)+(1-\alpha)J(m+2,t).
\end{equation}

Suppose first that $|P|<t^2$.
For each $u\in\{m,m+2\}$, by \eqref{eq:centralw}
$$
 J(u,t)\ge\frac{\sqrt u}{2}\ge2t>2\sqrt{|P|}.
$$
It follows that $\sqrt{|P|} <\frac{1}{2} (\alpha J(m,t) +(1-\alpha)J(m+2,t))$.
Consequently, by
\eqref{eq:tails-two-point-profile} and
\eqref{ineq-IPf},
$$
 \begin{aligned}
 I_{P^c}[f]
 =I[f]-I_P[f]
 &\ge \alpha J(m,t)+(1-\alpha)J(m+2,t)-\sqrt{|P|}\\
 &\ge \frac12\bigl[\alpha J(m,t)+(1-\alpha)J(m+2,t)\bigr]\\
 &\ge \frac12\Bigl[
   \alpha\theta(m,t)(m-|P|)_+
   +(1-\alpha)\theta(m+2,t)(m+2-|P|)_+
 \Bigr],
 \end{aligned}
$$
proving \eqref{eq:tails} in this case. Thus we may assume $|P|\geq t^2$.

Let $m_0$ be the largest admissible integer at most $|P|$, so that
$|P|\in\{m_0,m_0+1\}$, and set
$$
 c_P=
 \begin{cases}
  2(t-2),&|P|=m_0,\\
  2(t-1),&|P|=m_0+1.
 \end{cases}
$$

\begin{claim}\label{cl:partial-affine}
For every increasing Boolean $f$, and every $P$ with $|P|\ge t$,
\begin{equation}\label{eq:partial-affine}
 I_P[f]\le
 J(m_0,t)+c_P\bigl(\E f-w(m_0,t)\bigr).
\end{equation}
\end{claim}
\begin{proof}
By \eqref{ineq-Iif},
$$
 I_P[f]-c_P\E f
 =4\E\left[
 \left(\sum_{i\in P}X_i-\frac{|P|}{2}-\frac{c_P}{4}\right)f
 \right].
$$
To maximize this expectation over $0\le f\le1$, it suffices
to take $f=1$ where its coefficient is positive and $f=0$
where its coefficient is negative.

If $|P|=m_0$, then
$$
 \sum_{i\in P}X_i-\frac{|P|}{2}-\frac{c_P}{4}=\sum_{i\in P}X_i-\frac{m_0+t}{2}+1.
$$
Since $\sum_{i\in P}X_i$ is an integer, the maximum  of $I_P[f]-c_P\E f$ is
attained by  letting $f=\mathbbm{1}_{\{\sum_{i\in P}X_i\ge(m_0+t)/2\}}$.

If $|P|=m_0+1$, then
$$
  \sum_{i\in P}X_i-\frac{|P|}{2}-\frac{c_P}{4}=\sum_{i\in P}X_i-\frac{m_0+t}{2}.
$$
Let $\max P=j$.
In this case, the maximum  of $I_P[f]-c_P\E f$ is
attained by letting $f=\mathbbm{1}_{\{
 \sum_{i\in P\setminus\{j\}}X_i\ge(m_0+t)/2
 \}}$.

Both indicator functions are increasing, have mean
$w(m_0,t)$, and have total influence $J(m_0,t)$ on $P$.
Thus,
$$
 I_P[f]-c_P\E f\le J(m_0,t)-c_Pw(m_0,t),
$$
and \eqref{eq:partial-affine} follows.
\end{proof}

It remains to prove \eqref{eq:tails} when $|P|\ge t^2$.
If $|P|\ge m+2$, then the  right-hand side of \eqref{eq:tails} is zero and there is nothing to prove.
Thus we may assume $|P|\leq m+1$ and $m_0\le m$.

For any admissible $u\ge m_0$, define
$$
 G(u)=J(u,t)-J(m_0,t)
      -c_P\bigl(w(u,t)-w(m_0,t)\bigr).
$$
Using $\E f=\alpha w(m,t)+(1-\alpha)w(m+2,t)$, subtracting
\eqref{eq:partial-affine} from
\eqref{eq:tails-two-point-profile} we obtain
\begin{equation}\label{eq:tails-two-point-gap}
 I_{P^c}[f]\ge
 \alpha G(m)+(1-\alpha)G(m+2).
\end{equation}

\begin{claim}\label{cl:tail-gap}
For every admissible $v\ge m_0$,
\begin{equation}\label{eq:tails-endpoint-gap}
 G(v)\ge\frac12\theta(v,t)(v-|P|)_+.
\end{equation}
\end{claim}
\begin{proof}
For $v=m_0$, both sides are zero.
Suppose that $v>m_0$, so that $v\ge m_0+2>|P|$.
For every admissible $u>m_0$,
by \eqref{eq:incrementa} and \eqref{eq:incrementJ}
$$
 G(u)-G(u-2)=
 \begin{cases}
  \displaystyle
  \left(1+\frac{t-1}{u-1}\right)\theta(u,t),
       &|P|=m_0,\\[6pt]
  \theta(u,t),&|P|=m_0+1.
 \end{cases}
$$
In particular, $G(u)-G(u-2) \geq \theta(u,t)$.
Moreover,
\begin{equation}\label{eq:w-ratio}
 \frac{\theta(u,t)}{\theta(u-2,t)}
 =
 \frac{(u-1)(u-2)}{(u+t-2)(u-t)}
 \le1
 \qquad\text{whenever }u\ge(t-1)^2+1.
\end{equation}
Since $|P|\ge t^2$, every admissible $u>m_0$ satisfies
$$
 u\ge m_0+2\ge |P|+1\ge t^2+1.
$$
Thus,
\begin{align*}
G(v)= G(v)-G(m_0)&\geq \theta(v,t)+\theta(v-2,t)+\cdots+\theta(m_0+2,t)\\
&\ge\frac{v-m_0}{2}\theta(v,t)\ge\frac12\theta(v,t)(v-|P|)_+.\qedhere
\end{align*}
\end{proof}

By \eqref{eq:tails-endpoint-gap} and
\eqref{eq:tails-two-point-gap}, we conclude that
\begin{equation*}
 I_{P^c}[f]\ge
 \frac12\Bigl[
   \alpha\theta(m,t)(m-|P|)_+
   +(1-\alpha)\theta(m+2,t)(m+2-|P|)_+
 \Bigr].\qedhere
\end{equation*}
\end{proof}

\section{Small cross-total-influence of a counterexample}\label{sec:overlap}

Unless stated otherwise, $f=\mathbbm{1}_{\cA}$ and $g=\mathbbm{1}_{\cB}$ are the
oppositely shifted indicators supplied by Lemma~\ref{lem:closure}.
Their measures are at most $1/2$, and
\begin{equation}\label{eq:pair}
 \mu(\cA\cap\cB^c)=\mu(\cA)\,\mu(\cB)-\Cov(f,g^*).
\end{equation}

\begin{lemma}\label{lem:small-overlap}
Suppose $p,q\ge2$, $n\ge2^{16}(p+q-2)^2$, and
$\mu(\cA\cap\cB^c)>B_n$. Then
\begin{equation}\label{eq:small-overlap}
 0<\mathcal{I}[f,g]
 <n\exp\left(2-\frac{11\sqrt n}{40(p+q-2)}\right).
\end{equation}
\end{lemma}

The next lemma provides an upper bound for $\Cov(f,g^*)$ under the
counterexample assumption. Lemma~\ref{lem:correlation}, applied
to $f,g^*$, leads to a lower bound on the same covariance.

\begin{lemma}\label{lem:overlap}
Let $p,q\ge2$ and $n\ge2^{16}(p+q-2)^2$. Suppose that $\cA\subset 2^{[n]}$ is an increasing $p$-intersecting family,
$\cB\subset 2^{[n]}$ is an increasing $q$-intersecting family satisfying
$\mu(\cA\cap\cB^c)>B_n$. Then
\begin{equation}\label{eq:overlap}
 \Cov(f,g^*)<\frac{10(p+q-2)}{11\sqrt n}\,\mathcal{I}[f,g].
\end{equation}
\end{lemma}
\begin{proof}
By \eqref{eq:pair} and the assumption $\mu(\cA\cap\cB^c)>B_n$,
$\Cov(f,g^*)<\mu(\cA)\,\mu(\cB)-B_n$.
It therefore suffices to prove
\begin{equation}\label{eq:product-influence}
 \mu(\cA)\,\mu(\cB)-B_n
 \le\frac{10(p+q-2)}{11\sqrt n}\,\mathcal{I}[f,g].
\end{equation}
Since $f$, $g^*$ are both increasing, by the Harris--Kleitman lemma $\mu(\cA)\,\mu(\cB)>B_n$.
Since $\mu(\cA),\mu(\cB)\le1/2$, by \eqref{ineq-BnlowerB} we  have
\begin{align}\label{ineq-new4.1}
 \min\{\mu(\cA),\mu(\cB)\} > 2B_n\ge\frac12-\frac{2(p+q-2)}{\sqrt n}
       \ge\frac12-\frac1{128}.
\end{align}

Note that $w(m,p)$ is a non-decreasing function of $m$. By Theorem~\ref{thm-katona}, $\mu(\cA)\le w(n-\eps_1,p)$ and $\mu(\cB)\le w(n-\eps_2,q)$, where $\eps_1,\eps_2\in\{0,1\}$ satisfy $\eps_1\equiv n-p\pmod2$ and $\eps_2\equiv n-q\pmod2$.   Thus we may assume that
$$
 \mu(\cA)=\alpha w(a,p)+(1-\alpha)w(a+2,p),
 \qquad
 \mu(\cB)=\beta w(b,q)+(1-\beta)w(b+2,q),
$$
where $a,b$ are admissible for $p,q$, respectively,
$a+2,b+2\le n$, and $\alpha,\beta\in[0,1]$.

Since $\mu(\cA)\le w(a+2,p)$, by \eqref{ineq-new4.1}
$$
 \frac12-w(a+2,p)
 \le\frac12-\mu(\cA)
 <\frac1{128}<\frac1{96}.
$$
By~\eqref{eq:closer-half-dimension},
$a+2>256p^2+2$ and thereby $a>256p^2$.
Similarly, $b>256q^2$.
Thus one can apply \eqref{eq:tails} to both $f$ and $g$.

Relabel the coordinates, if necessary, so that
$$
 I_1[g]\le I_2[g]\le\cdots\le I_n[g].
$$
Using summation by parts,
$$
 \begin{aligned}
 \mathcal{I}[f,g]
 =I_1[g]I[f]+
 \sum_{k=1}^{n-1}
 \bigl(I_{k+1}[g]-I_k[g]\bigr)
 I_{[k+1,n]}[f].
 \end{aligned}
$$
For each $0\le k\le n-1$, by \eqref{eq:tails} we have
$$
 I_{[k+1,n]}[f]
 \ge\frac12\Bigl[
   \alpha\theta(a,p)(a-k)_+
   +(1-\alpha)\theta(a+2,p)(a+2-k)_+
 \Bigr].
$$
Using
$$
 \begin{aligned}
 &aI_1[g]+
 \sum_{k=1}^{n-1}
 \bigl(I_{k+1}[g]-I_k[g]\bigr)(a-k)_+=
 aI_1[g]+
 \sum_{k=1}^{a-1}(a-k)
 \bigl(I_{k+1}[g]-I_k[g]\bigr)=\sum_{i=1}^{a}I_i[g],
 \end{aligned}
$$
and the same identity with $a$ replaced by $a+2$, we obtain
$$
 \begin{aligned}
 \mathcal{I}[f,g]
 \ge\frac12\Biggl[
 &\alpha\theta(a,p)\sum_{i=1}^{a}I_i[g]+(1-\alpha)\theta(a+2,p)
                    \sum_{i=1}^{a+2}I_i[g]
 \Biggr].
 \end{aligned}
$$

Applying  \eqref{eq:tails} to $g$ with
$P=[a+1,n]$ and
$P=[a+3,n]$, we obtain that
$$
 \begin{aligned}
 \sum_{i=1}^{a}I_i[g]
 &\ge\frac12\Bigl[
   \beta\theta(b,q)(a+b-n)_+ +(1-\beta)\theta(b+2,q)(a+b+2-n)_+
 \Bigr],
 \end{aligned}
$$
and
$$
 \begin{aligned}
 \sum_{i=1}^{a+2}I_i[g]
 &\ge\frac12\Bigl[
   \beta\theta(b,q)(a+b+2-n)_+
   +(1-\beta)\theta(b+2,q)(a+b+4-n)_+
 \Bigr].
 \end{aligned}
$$
Substituting these two bounds, we conclude that
\begin{equation}\label{eq:abel-overlap}
 \begin{aligned}
 \mathcal{I}[f,g]
\ge&\frac14\Bigl[
\alpha\beta\theta(a,p)\theta(b,q)(a+b-n)_++\alpha(1-\beta)\theta(a,p)\theta(b+2,q)
                         (a+b+2-n)_+\\
 &+(1-\alpha)\beta\theta(a+2,p)\theta(b,q)
                         (a+b+2-n)_+\\
                         &+(1-\alpha)(1-\beta)\theta(a+2,p)\theta(b+2,q)
                         (a+b+4-n)_+
 \Bigr].
 \end{aligned}
\end{equation}

Combining \eqref{eq:abel-overlap} and
\eqref{eq:overlap-endpoint},
$$
 \begin{aligned}
 \mathcal{I}[f,g]
 &\ge\frac{11\sqrt n}{10(p+q-2)}\Bigl[
 \alpha\beta w(a,p)w(b,q)+\alpha(1-\beta)w(a,p)w(b+2,q)\\
 &\qquad+(1-\alpha)\beta w(a+2,p)w(b,q)\\
 &\qquad+(1-\alpha)(1-\beta)
 w(a+2,p)w(b+2,q)-B_n\Bigr]\\
 &=\frac{11\sqrt n}{10(p+q-2)}
 \bigl(\mu(\cA)\,\mu(\cB)-B_n\bigr).
 \end{aligned}
$$
Rearranging proves \eqref{eq:product-influence}, and hence
\eqref{eq:overlap}.
\end{proof}

We need the following correlation inequality, which is a weak version of Talagrand's correlation inequality \cite{Ta}.
However, it is useful for our proof because the bound has explicit constant.

\begin{lemma}[A quantitative correlation bound]\label{lem:correlation}
For increasing Boolean functions $f,g$,
\begin{equation}\label{eq:correlation}
 \Cov(f,g)\ge
 \frac{\mathcal{I}[f,g]}
 {4\bigl(2+\ln(n/\mathcal{I}[f,g])\bigr)}.
\end{equation}
The right-hand side is defined to be zero when $\mathcal{I}[f,g]=0$.
\end{lemma}
\begin{proof}
By \eqref{eq:covariance-integral}, it suffices to bound the integrand.
Since $f,g$ are increasing and Boolean, $D_i f$ and $D_i g$ are
nonnegative Boolean functions with means $I_i[f],I_i[g]$.
Thus by  \eqref{eq:classical-reverse} we have  for $0<\theta<1$,
$$
 \E[D_i f\,T_\theta D_i g]
 \ge\bigl(I_i[f]I_i[g]\bigr)^{1/(1-\theta)}.
$$
For $I_i[f]I_i[g]>0$, a change of variables and
$(1+\theta)^2\le e^{2\theta}$ yield
\begin{equation}\label{eq:integrated-reverse}
 \begin{aligned}
 \int_0^1\bigl(I_i[f]I_i[g]\bigr)^{1/(1-\theta)}\,d\theta
 =\int_0^\infty
       \frac{\bigl(I_i[f]I_i[g]\bigr)^{1+\theta}}
            {(1+\theta)^2}\,d\theta\ge\frac{I_i[f]I_i[g]}{2+\ln(1/(I_i[f]I_i[g]))}.
 \end{aligned}
\end{equation}
The last expression is a convex function of $I_i[f]I_i[g]$,
extended continuously to zero with value zero.
Thus, by \eqref{eq:covariance-integral}, \eqref{eq:integrated-reverse}
and Jensen's inequality we conclude that
$$
 \begin{aligned}
 \Cov(f,g)
 \ge\frac14\sum_i
       \frac{I_i[f]I_i[g]}{2+\ln(1/(I_i[f]I_i[g]))}\ge
 \frac{\mathcal{I}[f,g]}
 {4\bigl(2+\ln(n/\mathcal{I}[f,g])\bigr)}.
 \end{aligned}
$$
This proves \eqref{eq:correlation}.
\end{proof}
\begin{proof}[Proof of Lemma~\ref{lem:small-overlap}]
By Harris--Kleitman we  have $\Cov(f,g^*)\ge0$, so the strict inequality
\eqref{eq:overlap} implies $\mathcal{I}[f,g]>0$.
Apply \eqref{eq:correlation} to $f,g^*$, using
$\mathcal{I}[f,g^*]=\mathcal{I}[f,g]$, and combine it with
\eqref{eq:overlap}:
$$
 \frac{\mathcal{I}[f,g]}
      {4\bigl(2+\ln(n/\mathcal{I}[f,g])\bigr)}
 \le\Cov(f,g^*)
 <\frac{10(p+q-2)}{11\sqrt n}\,\mathcal{I}[f,g].
$$
Dividing by $\mathcal{I}[f,g]$ and rearranging, we obtain that
$$
 2+\ln\frac{n}{\mathcal{I}[f,g]}
 >\frac{11\sqrt n}{40(p+q-2)},
$$
which implies \eqref{eq:small-overlap}.
\end{proof}

\section{Localization near a product}\label{sec:localization}

\subsection{Smoothing estimates used in localization}\label{subsec:smoothing}

The next two lemmas estimate individual terms in
\eqref{eq:covariance-integral} when one function is close to the indicator of a Katona family. Lemma~\ref{lem:smoothed-slice} bounds the
fluctuation of its derivative after smoothing. Lemma~\ref{lem:smoothed-threshold}
then gives a lower bound proportional to
$I_i[f]/\sqrt{1+\ln(1/I_i[f])}$. We use its constant $64$ to obtain
the bound $\max\{\delta_f,\delta_g\}\le1/(64n)$ in
Proposition~\ref{prop:localization}. Its improved constant $24$ then
implies $2\delta_f/\theta(n_1,p),2\delta_g/\theta(n_2,q)\le e^{-L}$,
which are needed in Section~\ref{sec:completion}.

\begin{lemma}\label{lem:smoothed-slice}
Let $h$ be the indicator function of $\cK(n,P,t)$ or its Boolean dual, where
$m=|P|$ is admissible for $t\ge2$, and $m\ge256t^2+2$. For every $i\in P$,
\begin{equation}\label{eq:smoothed-slice}
 \|T_{1/2}D_i h-\theta(m,t)\|_2
 \le\theta(m,t)/5.
\end{equation}
\end{lemma}

\begin{proof}
It suffices to consider $h=\mathbbm{1}_{\cK(n,P,t)}$.
Indeed, its Boolean dual $h^*(x)=1-h(1-x)$ satisfies
 $D_i h^*(x)=D_i h(1-x)$.
Since the noise operator $T_{1/2}$ commutes with coordinate complementation,
$$
 \|T_{1/2}D_i h^*-\theta(m,t)\|_2
 =
 \|T_{1/2}D_i h-\theta(m,t)\|_2.
$$
Thus the estimate for $h$ also holds for $h^*$.

By relabelling, assume that $P=[m]$ and $i=m$. Then
$$
 D_m h(x)
 =
 \mathbbm{1}_{\left\{\sum_{i=1}^{m-1}x_i=(m+t)/2-1\right\}},
 \qquad
 \widehat{D_m h}(\varnothing)=\theta(m,t).
$$
By symmetry, the Fourier coefficient of $D_m h$ at $S\subseteq[m-1]$
depends only on $|S|$.

\begin{claim}\label{cl:slice-recurrence}
For $2\le j\le m-1$,
\begin{align}\label{ineq-3.4}
 (m-j)\widehat{D_m h}([j])
 =
 -(t-1)\widehat{D_m h}([j-1])
 -(j-1)\widehat{D_m h}([j-2]).
\end{align}
and
\begin{align}\label{ineq-3.6}
\widehat{D_m h}(\{1\}) &=-\frac{t-1}{m-1}\theta(m,t),\\
\widehat{D_m h}([2]) &=\frac{(t-1)^2-(m-1)}{(m-1)(m-2)}\theta(m,t).
\end{align}

\end{claim}
The proof is given in Appendix~\ref{app:claims}.

Let $a_j= \binom{m-1}{j}^{1/2}
 |\widehat{D_m h}([j])|/\theta(m,t)$.
Multiplying \eqref{ineq-3.4} by $\binom{m-1}{j}^{1/2}/\theta(m,t)$ and
taking absolute values, we obtain
$$
 \begin{aligned}
 a_j
 &\le
 \frac{t-1}{\sqrt{j(m-j)}}a_{j-1}+
 \sqrt{\frac{(j-1)(m-j+1)}{j(m-j)}}a_{j-2}.
 \end{aligned}
$$
Since $m-1>256t^2$, we infer that
\begin{align}\label{ineq-3.5}
 a_j
 \le
 \frac1{8\sqrt{2j}}
 a_{j-1}+
 \frac{33}{32}\sqrt{\frac{j-1}{j}}a_{j-2}, \mbox{ for }2\le j\le(m-1)/2.
\end{align}

By \eqref{ineq-3.6},
$$
 \begin{aligned}
 a_1&=\frac{\sqrt{m-1}\,
       |\widehat{D_m h}(\{1\})|}{\theta(m,t)}
 =\frac{t-1}{\sqrt{m-1}}\le\frac1{16},\\
 a_2^2&=\binom{m-1}{2}
 \frac{|\widehat{D_m h}([2])|^2}{\theta(m,t)^2}
 =
 \frac{\bigl((t-1)^2-(m-1)\bigr)^2}
      {2(m-1)(m-2)}
 \le\frac12,\\
 a_3&\leq \frac1{8\sqrt{6}}
 a_{2}+
 \frac{33}{32}\sqrt{\frac{2}{3}}a_1
 \le
 \frac1{16\sqrt3}+\frac{33}{512}\sqrt{\frac23}
 <\frac18.
 \end{aligned}
$$
Note that by Cauchy--Schwarz
$$
 \frac{1}{8\sqrt{2j}} \sqrt{17/16}
 +\frac{33}{32}\sqrt{\frac{j-1}{j}}
 \le
 \sqrt{\frac{17}{2048}+\left(\frac{33}{32}\right)^2}
 <\frac{17}{16}.
$$
By induction in \eqref{ineq-3.5}, we obtain that
\begin{align}\label{ineq-3.8}
 \binom{m-1}{j}
 \frac{|\widehat{D_m h}([j])|^2}{\theta(m,t)^2}
 \le\left(\frac{17}{16}\right)^j,
 \qquad 0\le j\le\frac{m-1}{2}.
\end{align}

\begin{claim}\label{cl:slice-symmetry}
\begin{align}\label{ineq-3.7}
 |\widehat{D_m h}([m-1-j])|
 =
 |\widehat{D_m h}([j])|.
\end{align}
\end{claim}
The proof is given in Appendix~\ref{app:claims}.

For $(m-1)/2<j\le m-1$, by
\eqref{ineq-3.7}, \eqref{ineq-3.8} and
$\binom{m-1}{j}=\binom{m-1}{m-1-j}$,
$$
 \binom{m-1}{j}
 \frac{|\widehat{D_m h}([j])|^2}{\theta(m,t)^2}
 =
 \binom{m-1}{m-1-j}
 \frac{|\widehat{D_m h}([m-1-j])|^2}{\theta(m,t)^2}\\
\le
 \left(\frac{17}{16}\right)^{m-1-j}
 \le
 \left(\frac{17}{16}\right)^j.
$$
Thus \eqref{ineq-3.8} holds for every $0\le j\le m-1$.

Note that
$$
 \widehat{T_{1/2}D_m h}(S)
 =2^{-|S|}\widehat{D_m h}(S),
 \qquad
 \widehat{D_m h}(\varnothing)
 =\E[D_m h]=\theta(m,t).
$$
Since Fourier coefficients
depend only on the cardinality of the indexing set, by Parseval's identity
$$
 \begin{aligned}
 \|T_{1/2}D_m h-\theta(m,t)\|_2^2
 &=
 \sum_{\varnothing\ne S\subseteq[m-1]}
 4^{-|S|}|\widehat{D_m h}(S)|^2=
 \sum_{j=1}^{m-1}
 4^{-j}\binom{m-1}{j}
 |\widehat{D_m h}([j])|^2.
 \end{aligned}
$$
Thus,
$$
 \begin{aligned}
 \frac{\|T_{1/2}D_m h-\theta(m,t)\|_2^2}
      {\theta(m,t)^2}
 &=
 \sum_{j=1}^{m-1}
 4^{-j}\binom{m-1}{j}
 \frac{|\widehat{D_m h}([j])|^2}{\theta(m,t)^2}\\
 &\le
 \frac1{1024}+\frac1{32}+\frac1{4096}
 +\sum_{j=4}^{\infty}\left(\frac{17}{64}\right)^j\\
 &=
 \frac1{1024}+\frac1{32}+\frac1{4096}
 +\frac{(17/64)^4}{1-17/64}
 <\frac1{25}.
 \end{aligned}
$$
Taking square roots proves \eqref{eq:smoothed-slice}.
\end{proof}

\begin{lemma}
\label{lem:smoothed-threshold}
Let $h$ be the indicator function of $\cK(n,P,t)$ or its Boolean dual, where
$m=|P|$ is admissible for $t\ge2$, and $m\ge256t^2+2$.
Let $g$ be increasing and Boolean. If
$(2\E|g-h|)^{4/5}\le\theta(m,t)/4$, then for every increasing
Boolean $f$ and $i\in P$,
\begin{equation}\label{eq:smoothed-correlation}
 \frac14\int_0^1
 \langle D_i f,\, T_\theta D_i g\rangle \,d\theta
 \ge
 \frac{I_i[f]}{64\sqrt m\sqrt{1+\ln(1/I_i[f])}}.
\end{equation}
If $n\ge2^{18}$ and $\E|g-h|\le1/(64n)$, the denominator
$64$ in \eqref{eq:smoothed-correlation} can be replaced by $24$.
The right-hand side is defined to be zero when $I_i[f]=0$.
\end{lemma}

\begin{proof}
If $I_i[f]=0$, the conclusion is immediate.
Assume that $I_i[f]>0$.
Since $g,h$ are increasing and Boolean,
$$
 |D_i g-D_i h|\le1,
 \mbox{ and }
 \E|D_i g-D_i h|\le2\E|g-h|.
$$
By \eqref{eq:classical-hypercontractivity},
$$
 \|T_{1/2}(D_i g-D_i h)\|_2
 \le\|D_i g-D_i h\|_{5/4} \leq (\E |D_i g-D_i h|^{5/4})^{4/5}
 \le(2\E|g-h|)^{4/5}.
$$
By the semigroup identity
$T_\theta=T_{2\theta}T_{1/2}$ \cite{OD} and \eqref{eq:classical-hypercontractivity}, for $0\le 2\theta\le
 1/\sqrt{1+2\ln(1/I_i[f])}$ we have
\begin{align*}
 \|T_\theta D_i g-\theta(m,t)\|_{2+2\ln(1/I_i[f])} &= \|T_{2\theta}\left(T_{1/2}D_i g-\theta(m,t)\right)\|_{2+2\ln(1/I_i[f])}\\[3pt]
 &\overset{\eqref{eq:classical-hypercontractivity}}{\le}
 \|T_{1/2}D_i g-\theta(m,t)\|_2\\[3pt]
  &\leq
 \|T_{1/2}D_i h-\theta(m,t)\|_2+\|T_{1/2}(D_i g-D_i h)\|_2\\[3pt]
 &\overset{\eqref{eq:smoothed-slice}}{\leq}
 \frac{\theta(m,t)}5+(2\E|g-h|)^{4/5}.
\end{align*}
Since $D_i f$ is Boolean with mean $I_i[f]$ and
$$
 I_i[f]^{-1/(2+2\ln(1/I_i[f]))}
 \le\sqrt{e}<\frac53,
$$
by $\E |D_i f| =I_i[f]$ we obtain that for  $0\le \theta\le
 1/(2\sqrt{1+2\ln(1/I_i[f])})$
\begin{align*}
\langle D_i f,\, T_\theta D_i g\rangle&=I_i[f]\theta(m,t)
+\langle D_i f,\
       T_\theta D_i g-\theta(m,t)\rangle\\
&\ge I_i[f]\theta(m,t)
-\langle |D_i f|,
       |T_\theta D_i g-\theta(m,t)|\rangle \\
 &
 \geq I_i[f]\theta(m,t)
 -
 I_i[f]^{1-1/(2+2\ln(1/I_i[f]))}\|T_\theta D_i g-\theta(m,t)\|_{2+2\ln(1/I_i[f])}\\
 &\ge
 I_i[f]\theta(m,t)
 -
 I_i[f]^{1-1/(2+2\ln(1/I_i[f]))}
 \left(\frac{\theta(m,t)}5+(2\E|g-h|)^{4/5}\right)\\
 &\ge
 I_i[f]\left(
 \frac23\theta(m,t)-\frac53(2\E|g-h|)^{4/5}
 \right).
\end{align*}

Since $(2\E|g-h|)^{4/5}\le\theta(m,t)/4$,  for  $0\le \theta\le
 1/(2\sqrt{1+2\ln(1/I_i[f])})$ we have
$$
\langle D_i f,\, T_\theta D_i g\rangle
 \ge\frac14I_i[f]\theta(m,t)
 \overset{\eqref{eq:profile-lower-central}}{\ge}\frac{3I_i[f]}{16\sqrt m}.
$$
Since $\E[D_i f\,T_\theta D_i g]\ge0$ for all
$0\le\theta\le1$, by  $2\sqrt{1+2\ln(1/I_i[f])}
 \le3\sqrt{1+\ln(1/I_i[f])}$ we conclude that
$$
 \frac14\int_0^1
 \langle D_i f,\, T_\theta D_i g\rangle\,d\theta
 \ge
 \frac{3I_i[f]}
 {128\sqrt m\sqrt{1+2\ln(1/I_i[f])}}\ge
 \frac{I_i[f]}
 {64\sqrt m\sqrt{1+\ln(1/I_i[f])}}.
$$

If $n\ge2^{18}$ and $\E|g-h|\le1/(64n)$, then
$$
 \|T_{1/2}(D_i g-D_i h)\|_2
 \le(2\E|g-h|)^{4/5}
 \le\frac1{512\sqrt n}
 \overset{\eqref{eq:profile-lower-central}}{<}\frac{\theta(m,t)}{100}.
$$
Then for $0\le\theta\le 1/(2\sqrt{1+2\ln(1/I_i[f])})$,
$$
\langle D_i f,\, T_\theta D_i g\rangle
 \ge\frac{13}{20}I_i[f]\theta(m,t)
 \ge\frac{I_i[f]}{2\sqrt m}.
$$
Consequently,
\begin{equation*}
 \frac14\int_0^1
\langle D_i f,\, T_\theta D_i g\rangle\,d\theta
 \ge
 \frac{I_i[f]}
 {16\sqrt m\sqrt{1+2\ln(1/I_i[f])}}>
 \frac{I_i[f]}
 {24\sqrt m\sqrt{1+\ln(1/I_i[f])}}. \qedhere
\end{equation*}
\end{proof}

\subsection{Localization of a hypothetical counterexample}
\label{subsec:localization-proof}

\begin{proposition}\label{prop:localization}
Suppose $p,q\ge2$, $n\ge2^{15}(p+q-2)^3$ and $n-p-q$ is odd.
If $\mu(\cA\cap\cB^c)>B_n$, then there is a partition
$[n]=X\sqcup\{z\}\sqcup Y$, with $X=[n_1]$, $z=n_1+1$, and
$Y=[n_1+2,n]$, such that $n_1$ is admissible for $p$, $n_2=|Y|$ is
admissible for $q$, and
\begin{equation}\label{eq:block-ratio}
 \left(\frac{p-1}{2(q-1)}\right)^{2/3}
 \le\frac {n_1}{n_2}\le\left(\frac{2(p-1)}{q-1}\right)^{2/3},
 \qquad n_1,n_2\ge\frac{n}{2(p+q-2)^{2/3}}.
\end{equation}
Let
$\delta_f=\mu(\cA\triangle\cK(n,X,p))$,
$\delta_g=\mu(\cB\triangle\cK(n,Y,q))$ and
\begin{equation}\label{eq:rarity}
 L=\frac{n}{96(p+q-2)^2},
\end{equation}
we have
\begin{equation}\label{eq:localization}
 \max\{\delta_f,\delta_g\}\le\frac1{64n},
 \qquad
 \max\left\{\frac{2\delta_f}{\theta(n_1,p)},
             \frac{2\delta_g}{\theta(n_2,q)}\right\}\le e^{-L}.
\end{equation}
\end{proposition}

The bound $\max\{\delta_f,\delta_g\}\le1/(64n)$ permits the
constant $24$ in Lemma~\ref{lem:smoothed-threshold}.
The other bounds in \eqref{eq:localization} control the measures
of the exceptional families of restricted projections.
Indeed,
$$
 \bigl(\mu^{n_1}(\cA(y,Y\cup\{z\}))-w(n_1,p)\bigr)_+
 \le\mu^{n_1}\bigl(\cA(y,Y\cup\{z\})
                 \triangle\cK(n,X,p)_{\mid X}\bigr).
$$
Averaging over $y$ and applying Markov's inequality gives
\begin{equation}\label{eq:localization-purpose}
\mu^{n_2+1}\bigl(\{y\subseteq Y\cup\{z\}:\mu^{n_1}(\cA(y,Y\cup\{z\}))>w(n_1,p)+\theta(n_1,p)/2\}\bigr)\le\frac{2\delta_f}{\theta(n_1,p)}\le e^{-L}.
\end{equation}
The analogous estimate holds for $\cB$.
Theorem~\ref{thm:exterior} uses these estimates for the neighborhood
hypothesis \eqref{eq:multiradius-hypothesis}, and
Lemma~\ref{lem:excess-covariance} uses them in
\eqref{eq:rare-section-influence}.
The bound $\delta_f\le1/(64n)$ alone would give only
$1/(32n\theta(n_1,p))$ in \eqref{eq:localization-purpose},
not the required bound $e^{-L}$.

\begin{proof}
\emph{Step 1: majority projections and approximation by Katona families.}
Let $P=\{i:I_i[f]\ge I_i[g]\}$,
$\cA'=\Pi_P^*\cA$, and $\cB'=\Pi_{P^c}^*\cB$.
By Section~\ref{subsec:shifting}, $P$ is an initial interval.
Set $\eta=\mu(\cA\triangle\cA')+\mu(\cB\triangle\cB')$.
By Lemma~\ref{lem:projection}, Cauchy--Schwarz and
\eqref{eq:small-overlap}, we obtain that
\begin{equation}\label{eq:first-projection-error}
 \begin{aligned}
 \eta
 \le\frac12\sum_{i=1}^n\min\{I_i[f],I_i[g]\}
 \le\frac12\sqrt{n\mathcal{I}[f,g]}<\frac{e n}{2}
       \exp\left(-\frac{11\sqrt n}{80(p+q-2)}\right)
 <\frac1{8n^{3/2}},
 \end{aligned}
\end{equation}
where the last step is \eqref{eq:initiallog}.

By \eqref{eq:pair} and the Harris--Kleitman correlation inequality,
$\mu(\cA)\,\mu(\cB)>B_n$. Since both families have measure at most $1/2$, by
\eqref{ineq-BnlowerB} and \eqref{eq:first-projection-error} we have
\begin{equation}\label{ineq-4.2}
 \begin{aligned}
 \min\{\mu(\cA'),\mu(\cB')\}\geq 2B_n - \eta
 &>\frac12-\frac{2(p+q-2)}{\sqrt n}-\eta\\
 &>\frac12-\frac{5(p+q-2)}{2\sqrt n}
 \ge\frac12-\frac5{512}>\frac{12}{25}.
 \end{aligned}
\end{equation}
In particular, both projections are nonempty. Their intersection
properties imply $|P|\ge p$ and $|P^c|\ge q$.
Let $n_1,n_2$ be the largest admissible integers at most $|P|,|P^c|$,
respectively. As $n-p-q$ is odd, exactly one of these two dimensions
has been reduced by one, so $n_1+n_2=n-1$. Set
$X=[n_1]$, $Y=[n-n_2+1,n]$, and $z=n_1+1$.

By the Katona theorem,
$\mu(\cA')\le w(n_1,p)$ and $\mu(\cB')\le w(n_2,q)$.
The definition of $B_n$, the inequality $\mu(\cA)\,\mu(\cB)>B_n$, and the
fact that all four families $\cA,\cB,\cA',\cB'$  have measures at most $1/2$ yield
\begin{equation}\label{ineq-4.1}
 0\le B_n-w(n_1,p)w(n_2,q)
 \le B_n-\mu(\cA')\mu(\cB') \le \mu(\cA)\mu(\cB)-\mu(\cA')\mu(\cB')<\frac\eta2
 <\frac1{16n^{3/2}},
\end{equation}
which follows from
\begin{align*}
\mu(\cA)\mu(\cB)-\mu(\cA')\mu(\cB') &\leq  \mu(\cB)(\mu(\cA)-\mu(\cA'))+\mu(\cA')(\mu(\cB)-\mu(\cB'))\\[3pt]
&\leq \frac{1}{2} (|\mu(\cA)-\mu(\cA')|+|\mu(\cB)-\mu(\cB')|) \\[3pt]
&\leq \frac{\eta}{2}.
\end{align*}
Moreover, \eqref{ineq-4.2} implies
$1/2-w(n_1,p),1/2-w(n_2,q)<5/512<1/96$.
Thus \eqref{eq:closer-half-dimension} implies
$n_1>256p^2+2$ and $n_2>256q^2+2$.
Using \eqref{ineq-4.1} and \eqref{ineq-4.2}, we obtain
$$
 0\le w(n_1,p)-\mu(\cA')
 \le\frac{B_n-\mu(\cA')\mu(\cB')}{\mu(\cB')}
 <\frac{25}{24}\eta,
$$
and the analogous inequality for $\cB'$.
Applying Lemma~\ref{lem:stability} to the two projections, we obtain that
\[
\mu(\cA'\triangle\cK(n,X,p)) \leq \frac{n_1-1}{p-1}(w(n_1,p)-\mu(\cA')) \leq \frac{n_1-1}{p-1}\, \frac{25}{24}\eta.
\]
Similarly,
\[
\mu(\cB'\triangle\cK(n,Y,q)) \leq \frac{n_2-1}{q-1}(w(n_2,q)-\mu(\cB')) \leq \frac{n_2-1}{q-1}\, \frac{25}{24}\eta.
\]
Consequently,
$$
 \delta_f= \mu(\cA\triangle \cK(n,X,p))\le \mu(\cA\triangle \cA')+\mu(\cA'\triangle\cK(n,X,p)) \leq \left(1+\frac{25(n_1-1)}{24(p-1)}\right)\eta
 \le\frac{13n_1}{12(p-1)}\eta,
$$
and similarly
\[
 \delta_g\le\frac{13n_2}{12(q-1)}\eta.
\]
Together with \eqref{eq:first-projection-error}, this gives
\begin{equation}\label{eq:first-distance}
 \delta_f<\frac{3nn_1}{2(p-1)}e^{-11\sqrt n/[80(p+q-2)]},
 \qquad
 \delta_g<\frac{3nn_2}{2(q-1)}e^{-11\sqrt n/[80(p+q-2)]}.
\end{equation}

\emph{Step 2: comparison of the block sizes.}
Suppose $n_1/n_2<((p-1)/(2(q-1)))^{2/3}$.
Moving two coordinates from the second block to the first, and using
\eqref{eq:incrementa}, \eqref{eq:profile-upper-sharp}, and
\eqref{eq:profile-lower-central} we obtain that
\begin{equation}\label{eq:block-transfer}
 \begin{aligned}
 \frac1{16n^{3/2}}
 &>B_n-w(n_1,p)w(n_2,q)\\
 &\ge w(n_1+2,p)w(n_2-2,q)-w(n_1,p)w(n_2,q)\\
  &= (w(n_1+2,p)-w(n_1,p))w(n_2-2,q)-w(n_1,p)(w(n_2,q)-w(n_2-2,q))\\
   &\ge \frac{(p-1)\theta(n_1+2,p)}{2(n_1+1)}w(n_2-2,q)-w(n_1,p)\frac{(q-1)\theta(n_2,q)}{2(n_2-1)}\\
 &\ge\frac{4(p-1)}{25n_1^{3/2}}-\frac{q-1}{4n_2^{3/2}}
 >\frac{7(q-1)}{100n_2^{3/2}}>\frac1{16n^{3/2}},
 \end{aligned}
\end{equation}
a contradiction. In the five line we used $w(n_2-2,q)\ge7/16$
and $(n_1+2)^{3/2}\le17n_1^{3/2}/16$, which follow from $n_1>256p^2+2$. Exchanging the blocks proves both ratio bounds in
\eqref{eq:block-ratio}. By H\"older's inequality,
$$
 1+\frac {n_2}{n_1}\le1+\bigl(2(p+q-3)\bigr)^{2/3}
 \le5^{1/3}(p+q-2)^{2/3}<\frac74(p+q-2)^{2/3}.
$$
Since $n_1+n_2=n-1$, this implies the lower bound for $n_1$ in
\eqref{eq:block-ratio}; the bound for $n_2$ is symmetric.

If $p+q-2=2$, \eqref{eq:block-ratio} implies $n_1,n_2<5n/8$.
Thus by \eqref{eq:first-distance}, \eqref{eq:coarse-error-number} and
\eqref{eq:profile-lower-central} we obtain that
\begin{equation}\label{eq:initial-smoothing-error}
 (2\delta_f)^{4/5}\le\frac3{16\sqrt {n_1}}\le\frac{\theta(n_1,p)}4,
 \qquad
 (2\delta_g)^{4/5}\le\frac3{16\sqrt {n_2}}\le\frac{\theta(n_2,q)}4.
\end{equation}

\emph{Step 3: bounds for $I_Y[f]+I_X[g]$ and $\delta_f+\delta_g$.}
Subtract \eqref{eq:partial-affine}, with $P=X\cup\{z\}$, from the
inequality \eqref{eq:profile-supporting},
with $m=n_1+2$ and $t=p$. By~\eqref{eq:incrementa} and
\eqref{eq:incrementJ},
\begin{equation}\label{eq:exterior-linear-bound}
 \mu(\cA)-w(n_1,p)\le\frac{p-1}{2(n_1+1)}I_Y[f],
 \qquad
 \mu(\cB)-w(n_2,q)\le\frac{q-1}{2(n_2+1)}I_X[g].
\end{equation}
Set
\begin{equation}\label{eq:weighted-exterior-influence}
 \Gamma=\frac{p-1}{4(n_1+1)}I_Y[f]+\frac{q-1}{4(n_2+1)}I_X[g].
\end{equation}
Since $\mu(\cA)\,\mu(\cB)>B_n\ge w(n_1,p)w(n_2,q)$ and all four measures $\mu(\cA),\mu(\cB), w(n_1,p), w(n_2,q)$ are at most
$1/2$, \eqref{eq:pair} and \eqref{eq:exterior-linear-bound} imply
\[
\begin{aligned}\label{eq:covariance-weighted-upper}
 \Cov(f,g^*)&<\mu(\cA)\,\mu(\cB)-w(n_1,p)w(n_2,q)\\[3pt]
 &\leq \frac{1}{2}(\mu(\cB)+w(n_2,q))(\mu(\cA)-w(n_1,p))+\frac{1}{2}(w(n_1,p)+\mu(\cA))(\mu(\cB)-w(n_2,q))\\[3pt]
  &\leq \frac{1}{2}(\mu(\cA)-w(n_1,p))+\frac{1}{2}(\mu(\cB)-w(n_2,q))\\[3pt]
 &\leq \Gamma.
\end{aligned}
\]
In particular, $\Gamma>0$.

By \eqref{eq:initial-smoothing-error}, we may apply
\eqref{eq:smoothed-correlation} to $f,g^*$ for $i\in Y$,
and to $g^*,f$ for $i\in X$.
Since $I_i[g^*]=I_i[g]$ and $T_\theta$ is self-adjoint,
\[
 \langle D_i f, T_\theta D_i g^*\rangle
 =\langle T_\theta D_i f, D_i g^*\rangle=
 \langle D_i g^*, T_\theta D_i f\rangle.
\]
Consequently,
\[
 \frac14\int_0^1
 \langle D_i f, T_\theta D_i g^*\rangle\,d\theta
 \ge
 \begin{cases}
 \displaystyle
 \frac{I_i(f)}
 {64\sqrt n_2\sqrt{1+\ln(1/I_i(f))}},
 &i\in Y,\\[8pt]
 \displaystyle
 \frac{I_i(g)}
 {64\sqrt n_1\sqrt{1+\ln(1/I_i(g))}},
 &i\in X.
 \end{cases}
\]
Thus,
$$
 \Cov(f,g^*)\ge\frac1{64\sqrt {n_2}}
    \sum_{i\in Y}\frac{I_i[f]}{\sqrt{1+\ln(1/I_i[f])}}
 +\frac1{64\sqrt {n_1}}
    \sum_{i\in X}\frac{I_i[g]}{\sqrt{1+\ln(1/I_i[g])}}.
$$
The function $u\mapsto u/\sqrt{1+\ln(1/u)}$, extended by zero at
$u=0$, is convex and increasing on $[0,1]$.
By \eqref{eq:geometric},
$$
 \frac1{64\sqrt {n_2}}\ge
 \frac{\sqrt n}{18(p+q-2)}\frac{p-1}{4(n_1+1)},
$$
and the exchanged inequality also holds.
Apply Jensen with weights $(p-1)/(4n(n_1+1))$ on the $Y$ coordinates
and $(q-1)/(4n(n_2+1))$ on the $X$ coordinates, putting the remaining
weight at zero. Their sum is at most one and their weighted mean is
$\Gamma/n$. We obtain
\begin{equation}\label{eq:weighted-smoothing}
 \Cov(f,g^*)\ge
 \frac{\sqrt n}{18(p+q-2)}
 \frac{\Gamma}{\sqrt{1+\ln(n/\Gamma)}}.
\end{equation}
Combining \eqref{eq:covariance-weighted-upper} and
\eqref{eq:weighted-smoothing}, and using that both coefficients in
\eqref{eq:weighted-exterior-influence} are at least $1/(4n)$, yields
\begin{equation}\label{eq:weak-exterior-influence}
 I_Y[f]+I_X[g]\le4n\Gamma
 <4n^2\exp\left(1-\frac{n}{18^2(p+q-2)^2}\right).
\end{equation}

In view of Lemma~\ref{lem:projection},
\[
\begin{aligned}
 \mu(\cA\triangle\Pi_{X\cup\{z\}}^*\cA)
 \le\frac12 I_Y[f], \quad
 \mu(\cB\triangle\Pi_{Y\cup\{z\}}^*\cB)
 \le\frac12 I_X[g].
\end{aligned}
\]
By \eqref{eq:weak-bootstrap-number} and  $\mu(\cA)\,\mu(\cB)>B_n$,
\[
w(n_1,p)\geq  \mu(\Pi_{X\cup\{z\}}^*\cA)\geq \mu(\cA) -\frac12 I_Y[f] > 2B_n -4n^2\exp\left(1-\frac{n}{18^2(p+q-2)^2}\right) >\frac{1}{4}.
\]
Similarly, $w(n_2,q)\geq\mu(\Pi_{Y\cup\{z\}}^*\cB)>1/4$.
By the same product-deficit argument as in Step 1 one can obtain that
$$
 0\le w(n_1,p)-\mu(\Pi_{X\cup\{z\}}^*\cA)<I_Y[f]+I_X[g],
$$
and its counterpart for $\cB$.
By Lemma~\ref{lem:stability}, again with the order reversed for $\cB$,
\[
\begin{aligned}\label{eq:distance-from-exterior}
 \delta_f+\delta_g &\leq \mu(\cA\triangle\Pi_{X\cup\{z\}}^*\cA)+\mu(\Pi_{X\cup\{z\}}^*\cA)\triangle \cK(n,X,p)+\mu(\cB\triangle\Pi_{Y\cup\{z\}}^*\cB)+\mu(\Pi_{Y\cup\{z\}}^*\cB)\triangle \cK(n,Y,q)\\[3pt]
&\le\left(\frac12+\frac{n_1-1}{p-1}+\frac{n_2-1}{q-1}\right)
      \bigl(I_Y[f]+I_X[g]\bigr)\\[3pt]
&\le3n\bigl(I_Y[f]+I_X[g]\bigr).
\end{aligned}
\]
It follows from \eqref{eq:weak-exterior-influence} and
\eqref{eq:weak-bootstrap-number} that
\begin{equation}\label{eq:weak-localization}
 \delta_f+\delta_g
 <12n^3\exp\left(1-\frac{n}{18^2(p+q-2)^2}\right)
 \le\frac1{64n}.
\end{equation}

\emph{Step 4: the remaining inequalities in \eqref{eq:localization}.}
We can now use the constant $24$ in Lemma~\ref{lem:smoothed-threshold}.
Repeating the argument for \eqref{eq:weighted-smoothing} replaces its
factor $\sqrt n/(18(p+q-2))$ by $4\sqrt n/(27(p+q-2))$.
Thus
\begin{equation}\label{eq:bootstrap-log}
 1+\ln(n/\Gamma)>
 \frac{16n}{27^2(p+q-2)^2}>
 \frac{n}{48(p+q-2)^2}=2L.
\end{equation}
By \eqref{eq:distance-from-exterior},
$\delta_f+\delta_g\le12n^2\Gamma<12n^3e^{1-2L}$.
Finally, $\theta(n_1,p),\theta(n_2,q)\ge1/(2\sqrt n)$ and
\eqref{eq:finalnumbers} imply
$$
 \max\left\{\frac{2\delta_f}{\theta(n_1,p)},
             \frac{2\delta_g}{\theta(n_2,q)}\right\}
 \le4\sqrt n(\delta_f+\delta_g)
 <48 e n^{7/2}e^{-2L}\le e^{-L},
$$
because $48 e n^{7/2}\le(n+1)^6\le e^L$.
Together with \eqref{eq:weak-localization}, this proves
\eqref{eq:localization}.
\end{proof}

\section{Proof of the main theorem}\label{sec:completion}

Suppose first that $n-p-q$ is odd and that a counterexample exists.
Let $\cA,\cB$ and $f=\mathbbm{1}_{\cA}$,
$g=\mathbbm{1}_{\cB}$ be as in Lemma~\ref{lem:closure}, and take the
partition $[n]=X\sqcup Y\sqcup\{z\}$ from
Proposition~\ref{prop:localization}. Retain $n_1=|X|$, $n_2=|Y|$,
$\delta_f,\delta_g$ and $L$ from that proposition.
By \eqref{eq:pair}, it is enough to show that
$$
 \mu(\cA)\,\mu(\cB)-w(n_1,p)w(n_2,q)\le\frac23\Cov(f,g^*),
$$
since $w(n_1,p)w(n_2,q)\le B_n$ and $\Cov(f,g^*)\ge0$.

For $y\subseteq Y\cup\{z\}$ and $x\subseteq X\cup\{z\}$, define
$\cA_y=\cA(y,Y\cup\{z\})$ and $\cB_x=\cB(x,X\cup\{z\})$, and define
\begin{equation}\label{eq:excess-functions}
 \begin{aligned}
 \phi_\cA(y)=\left(\mu^{n_1}(\cA_y)-w(n_1,p)-\theta(n_1,p)/2\right)_+, \quad
 \phi_\cB(x)=\left(\mu^{n_2}(\cB_x)-w(n_2,q)-\theta(n_2,q)/2\right)_+.
 \end{aligned}
\end{equation}
The measures of these projections are taken on $2^X$ and $2^Y$,
respectively, whereas $\phi_\cA$ and $\phi_\cB$ are functions on
$2^{Y\cup\{z\}}$ and $2^{X\cup\{z\}}$.
For covariances, we also regard them as functions on $2^{[n]}$
independent of the remaining coordinates.

By \eqref{eq:localization-purpose} and its counterpart for $g$,
\begin{equation}\label{eq:excess-support}
 \mu^{n_2+1}(\phi_\cA>0)\le\frac{2\delta_f}{\theta(n_1,p)}\le e^{-L},
 \qquad
 \mu^{n_1+1}(\phi_\cB>0)\le\frac{2\delta_g}{\theta(n_2,q)}\le e^{-L}.
\end{equation}

In Sections~\ref{sec:exterior} and~\ref{subsec:rare-expansion}, we will prove
\eqref{eq:final-density-gain} and its counterpart for $g$.
More precisely, Theorem~\ref{thm:exterior} bounds
$(\mu(\cF)-w(n_1,t))_+$ by
$8(t-1)\E\phi/(n_1\zeta)$ for some $\zeta>0$, provided every increasing family of measure
at most $\rho$ on the coordinates outside $X$ satisfies
\eqref{eq:multiradius-hypothesis}.
Lemma~\ref{lem:rare} verifies this hypothesis for the cube
$2^{Y\cup\{z\}}$ with
$\rho=e^{-L}$, $\zeta=\sqrt L/(2\sqrt{n_2+1})$, and every
$i\in\{1,\ldots,p-1\}$.
Together with \eqref{eq:excess-support} and the numerical bounds in
Lemma~\ref{lem:numerical}, it follows that
$$
 (\mu(\cA)-w(n_1,p))_+
 \le\frac{16(p-1)\sqrt{n_2+1}}{n_1\sqrt L}\,\E\phi_\cA.
$$
Interchanging the two families $\cA$, $\cB$ yields the corresponding inequality
for $g$ on $2^{X\cup\{z\}}$.

In Section~\ref{subsec:excess-covariance}, we  prove
\eqref{eq:rare-correlation}, a lower bound for $\Cov(f,g^*)$ in terms
of the same quantities $\E\phi_\cA$ and $\E\phi_\cB$.
For this step, the bound $\max\{\delta_f,\delta_g\}\le1/(64n)$ from
Proposition~\ref{prop:localization} permits the constant $24$ in
Lemma~\ref{lem:smoothed-threshold}; the bounds
\eqref{eq:excess-support} permit the edge-isoperimetric estimate on
each positive level set of $\phi_\cA$ and $\phi_\cB$.
Finally, \eqref{eq:block-ratio} makes the coefficients in the two
upper bounds small enough to compare with this covariance lower
bound. The calculation is given in \eqref{eq:final-comparison}.

\subsection{A bound for the measure of an intersecting family}\label{sec:exterior}

We compare the restricted projections of $\cF$ with the Katona
family after applying Lemma~\ref{lem:simplicial-compression}.
The shadow inclusions below compare individual levels;
Lemma~\ref{lem:interval-coupling} combines their neighborhood bounds.

\Needspace{18\baselineskip}
\begin{theorem}\label{thm:exterior}
Let $[n]=X\sqcup Y\sqcup\{z\}$ and let
$\cF\subseteq2^{[n]}$ be increasing and $t$-intersecting, where $t\ge2$.
Suppose $n_1=|X|$ is admissible for $t$ and $n_1\ge4t$.
For $y\subseteq Y\cup\{z\}$, let $\cF_y=\cF(y,Y\cup\{z\})$ and set
$$
 \phi(y)=\left(\mu^{n_1}(\cF_y)-w(n_1,t)-\frac{\theta(n_1,t)}{2}\right)_+.
$$
All neighborhoods below are taken in $2^{Y\cup\{z\}}$,
with its uniform measure $\mu^{n-n_1}$. Suppose
$$
 \frac{2\mu^n(\cF\triangle\cK(n,X,t))}{\theta(n_1,t)}\le\rho<1,
 \qquad \frac{24(t-1)}{n_1}\le\zeta\le\frac1{16}.
$$
If every increasing family
$\cG\subseteq2^{Y\cup\{z\}}$ with $\mu^{n-n_1}(\cG)\le\rho$ satisfies
\begin{equation}\label{eq:multiradius-hypothesis}
 \mu^{n-n_1}(N_i(\cG))\ge(1+\zeta)^i\mu^{n-n_1}(\cG)
       \mbox { for all } i=1,2,\ldots, t-1,
\end{equation}
then
\begin{equation}\label{eq:exterior}
 (\mu^n(\cF)-w(n_1,t))_+\le\frac{8(t-1)}{n_1\zeta}\E\phi.
\end{equation}
\end{theorem}
\begin{proof}
Relabel $X=[n_1]$ and set $n_1=t+2r$.
Throughout this proof, $y^c=(Y\cup\{z\})\setminus y$,
complements of subsets of $[n_1]$ are taken in $[n_1]$, and
$\E_y$ averages over $y\subseteq Y\cup\{z\}$.
Apply Lemma~\ref{lem:simplicial-compression}, we assume that
$\cF=\mathcal C_X\cF$ and $\cF_y=\mathcal T_{[n_1]}(|\cF_y|)$ for all $y\subset Y\cup \{z\}$.

For convenient, let
\begin{equation}\label{eq:projection-comparison-families}
 \cK=\cK(n_1,[n_1],t),\qquad
 \cA_y=\cF_y\setminus\cK,\qquad
 \cB_y=\cK\setminus\cF_{y^c}.
\end{equation}
Note that both $\cF_y$ and $\cK$ consist of the last subsets in the same
simplicial order on $2^{[n_1]}$. Thus either $\cF_y\subseteq\cK$ or
$\cK\subseteq\cF_y$, and
\begin{equation}\label{eq:projection-surplus-deficit}
 \begin{aligned}
 \mu^{n_1}(\cA_y)=(\mu^{n_1}(\cF_y)-w(n_1,t))_+,\quad
 \mu^{n_1}(\cB_y)=(w(n_1,t)-\mu^{n_1}(\cF_{y^c}))_+.
 \end{aligned}
\end{equation}
If $y\subseteq y'$, then the increasingness of $\cA,\cB$ implies $\cA_y\subseteq\cA_{y'}$ and
$\cB_y\subseteq\cB_{y'}$.
Moreover,
$\mu^{n_1}(\cF_y)-w(n_1,t)=\mu^{n_1}(\cA_y)-\mu^{n_1}(\cB_{y^c})$.
Taking expectation on both sides,  we get
\begin{equation}\label{eq:projection-mean-balance}
 \begin{aligned}
 \mu^n(\cF)-w(n_1,t)
 =\E_y\mu^{n_1}(\cA_y)-\E_y\mu^{n_1}(\cB_y).
 \end{aligned}
\end{equation}
Note that
\[
 \phi(y)=\left(\mu^{n_1}(\cA_y)-\frac{\theta(n_1,t)}{2}\right)_+.
\]
Since $\E_y\mu^{n_1}(\cA_y)\le
\mu^n(\cF\triangle\cK(n,[n_1],t))$, by Markov's inequality
\begin{equation}\label{eq:projection-rare-support}
 \mu^{n-n_1}(\phi>0)
 \le\frac{2\E_y\mu^{n_1}(\cA_y)}{\theta(n_1,t)}\le\rho.
\end{equation}

Define
\begin{equation}\label{eq:shadow-layer-sums}
 \begin{aligned}
 q_j&=2^{-n_1}\binom {n_1}{r+t-1-j}
       && \mbox{ for  }0\le j\le r+t-1.
 \end{aligned}
\end{equation}
Set $q_j=0$ outside $0\le j\le r+t-1$.
If $\cA_y^{(r+t-1-j)}\ne\varnothing$, all preceding levels
$r+t-1,\ldots,r+t-j$ of $\cA_y$ are complete.
Likewise, if $\cB_y^{(r+t+j)}\ne\varnothing$, its levels
$r+t,\ldots,r+t+j-1$ are complete.
In particular,
\begin{equation}\label{eq:layer-event-meaning}
 \mu^{n_1}(\cA_y^{(r+t-1)})=\min\{\mu^{n_1}(\cA_y),q_0\},\qquad
 \mu^{n_1}(\cB_y^{(r+t)})=\min\{\mu^{n_1}(\cB_y),q_{t-1}\}.
\end{equation}
Figure~\ref{fig:katona-projections} illustrates the added and missing sets in these projections.

\begin{figure}[htbp]
\centering
\begin{tikzpicture}[x=1cm,y=1cm,font=\small,line width=.55pt]
\foreach \a in {0,4.6,9.2}{
 \begin{scope}[xshift=\a cm]
  \draw (0,2.2)--(1.65,0)--(0,-2.2)--(-1.65,0)--cycle;
  \node[above] at (0,2.2) {$[n_1]$};
  \node[below] at (0,-2.2) {$\varnothing$};
 \end{scope}
}
\begin{scope}
 \clip (0,2.2)--(1.65,0)--(0,-2.2)--(-1.65,0)--cycle;
 \fill[black!12] (-2,.25) rectangle (2,2.2);
 \draw[dashed] (-2,.25)--(2,.25);
 \draw (-1.5,.52)--(1.5,.52);
 \draw (-1.5,-.04)--(1.5,-.04);
\end{scope}
\node at (0,1.3) {$\cK$};
\node[right,font=\scriptsize] at (1.25,.72) {$r+t$};
\node[right,font=\scriptsize] at (1.62,-.1) {$r+t-1$};
\node[align=center] at (0,-3) {$|A|\ge r+t$\\defines $\cK$};
\begin{scope}[xshift=4.6cm]
 \begin{scope}
  \clip (0,2.2)--(1.65,0)--(0,-2.2)--(-1.65,0)--cycle;
  \fill[black!12] (-2,.25) rectangle (2,2.2);
  \fill[black!27] (-2,-.68) rectangle (2,.25);
  \fill[black!27] (.05,-.88) rectangle (2,-.68);
  \draw[dashed] (-2,.25)--(2,.25);
  \draw (-2,-.68)--(2,-.68);
  \draw (-2,-.88)--(2,-.88);
 \end{scope}
 \node at (0,1.3) {$\cK\subseteq\cF_y$};
 \node at (0,-.18) {$\cA_y$};
 \node[below,font=\scriptsize] at (.42,-.94) {level $r+t-1-j$};
 \node[align=center] at (0,-3) {$\cA_y=\cF_y\setminus\cK$};
\end{scope}
\begin{scope}[xshift=9.2cm]
 \begin{scope}
  \clip (0,2.2)--(1.65,0)--(0,-2.2)--(-1.65,0)--cycle;
  \fill[black!12] (-2,1.05) rectangle (2,2.2);
  \fill[black!12] (.05,.85) rectangle (2,1.05);
  \fill[pattern=north east lines,pattern color=black!55]
     (-2,.25) rectangle (2,.85);
  \fill[pattern=north east lines,pattern color=black!55]
     (-2,.85) rectangle (.05,1.05);
  \draw[dashed] (-2,.25)--(2,.25);
  \draw (-2,.85)--(2,.85);
  \draw (-2,1.05)--(2,1.05);
 \end{scope}
 \node[font=\scriptsize] at (0,1.53) {$\cF_{y^c}\subseteq\cK$};
 \node[fill=white,inner sep=1pt] at (0,.54) {$\cB_y$};
 \node[align=center] at (0,-3) {$\cB_y=\cK\setminus\cF_{y^c}$};
\end{scope}
\end{tikzpicture}
\caption{Three copies of $2^{[n_1]}$, ordered vertically by set size.
The dashed line separates $|A|\ge r+t$ from $|A|\le r+t-1$.
The middle and right drawings illustrate the case
$\cK\subseteq\cF_y$ and $\cF_{y^c}\subseteq\cK$.
The darker region is $\cA_y$; the hatched region is $\cB_y$ and
is absent from $\cF_{y^c}$.}
\label{fig:katona-projections}
\end{figure}
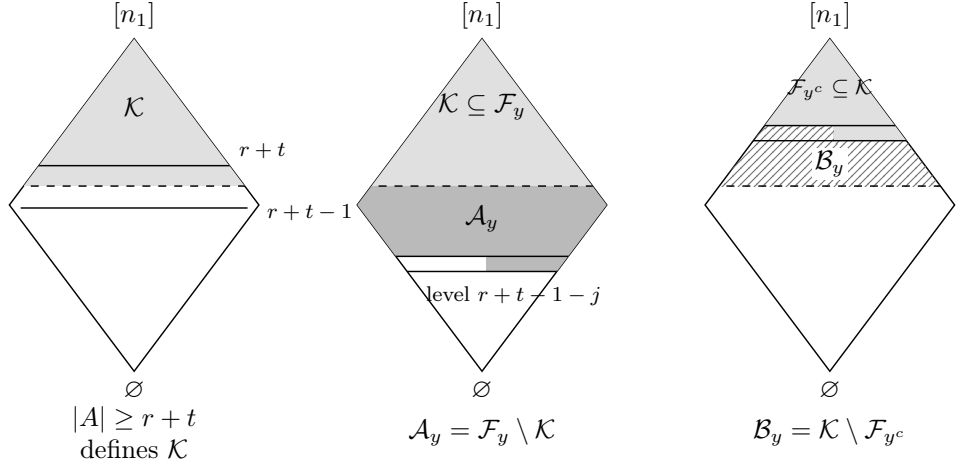

\begin{claim}\label{claim:projection}
The projections in~\eqref{eq:projection-comparison-families} have the following properties.
\begin{itemize}
\item[(i)] $\cF_y$ and $\cF_{y^c}$ are cross $t$-intersecting.
\item[(ii)] If $\cA_y^{(t-1)}\ne\varnothing$, then
$\cF_{y^c}=\varnothing$ and $\cB_y=\cK$; in particular,
$\mu^{n_1}(\cB_y)=w(n_1,t)$.
\item[(iii)] If $y\subseteq y'$, $|y'\setminus y|\le i$,
$0\le i\le t-1$, and $0\le j\le r$, then
\begin{equation}\label{eq:first-shadow-inclusion}
 \left(\partial^{t-1-i}
       \cA_{y'}^{(r+t-1-j-i)}\right)^c
 \subseteq\cB_y^{(r+t+j)}.
\end{equation}
In particular,
$(\partial^{t-1}\cA_y^{(r+t-1-j)})^c
\subseteq\cB_y^{(r+t+j)}$ for $0\le j\le r$.
\item[(iv)] For $0\le j\le r$,
$$
 \cA_y^{(r+t-1-j)}=\binom{[n_1]}{r+t-1-j}
 \mbox{ implies }
 \cB_y^{(r+t+j)}=\binom{[n_1]}{r+t+j}.
$$
Equivalently, $\mu^{n_1}(\cA_y^{(r+t-1-j)})=q_j$ implies
$\mu^{n_1}(\cB_y^{(r+t+j)})=q_{j+t-1}$.
\end{itemize}
\end{claim}
\begin{proof}[Proof of the claim~\ref{claim:projection}]
For $A\in\cF_y$ and $C\in\cF_{y^c}$,
by the $t$-intersection of $\cF$
$$
|A\cap C|=|(A\cup y)\cap(C\cup y^c)|\ge t.
$$
This proves (i), which also proves (ii) if $|A|=t-1$.

For (iii), take
$A\in\cA_{y'}^{(r+t-1-j-i)}$ and
$B\subseteq A$ with $|B|=r-j$. Then
\begin{equation}\label{eq:shadow-forbidden-intersection}
 |(A\cup y')\cap(([n_1]\setminus B)\cup y^c)|
 =t-1-i+|y'\setminus y|\le t-1.
\end{equation}
Thus $[n_1]\setminus B\notin\cF_{y^c}$. Its size is $r+t+j$, so it belongs to $\cB_y^{(r+t+j)}$; see Figure~\ref{fig:shadow-complements}. Finally, the $(t-1)$st lower shadow of the complete
$(r+t-1-j)$th level is the complete $(r-j)$th level.
Taking complements in (iii) with $i=0$ proves (iv).
\end{proof}
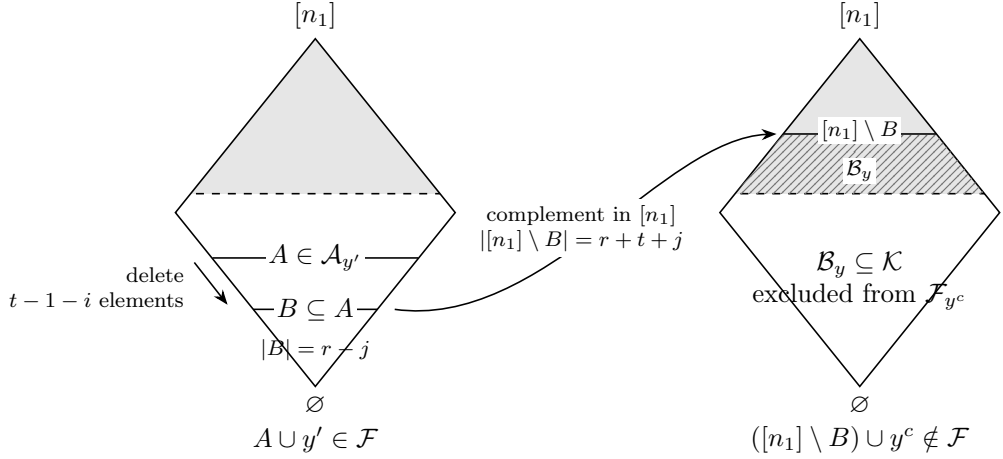
\begin{figure}[htbp]
\centering
\begin{tikzpicture}[x=1cm,y=1cm,font=\small,line width=.6pt,>=Stealth]
\foreach \a in {0,7.2}{
 \begin{scope}[xshift=\a cm]
  \begin{scope}
   \clip (0,2.3)--(1.85,0)--(0,-2.3)--(-1.85,0)--cycle;
   \fill[black!10] (-2,.25) rectangle (2,2.3);
   \draw[dashed] (-2,.25)--(2,.25);
  \end{scope}
  \draw (0,2.3)--(1.85,0)--(0,-2.3)--(-1.85,0)--cycle;
  \node[above] at (0,2.3) {$[n_1]$};
  \node[below] at (0,-2.3) {$\varnothing$};
 \end{scope}
}
\draw (-1.36,-.6)--(1.36,-.6);
\draw (-.82,-1.28)--(.82,-1.28);
\node[fill=white,inner sep=1pt] at (0,-.6) {$A\in\cA_{y'}$};
\node[fill=white,inner sep=1pt] at (0,-1.28) {$B\subseteq A$};
\draw[->] (-1.6,-.66)--(-1.16,-1.22);
\node[anchor=east,align=right,font=\scriptsize] at (-1.6,-1)
 {delete\\$t-1-i$ elements};
\node[align=center,font=\scriptsize] at (0,-1.82) {$|B|=r-j$};
\begin{scope}[xshift=7.2cm]
 \begin{scope}
  \clip (0,2.3)--(1.85,0)--(0,-2.3)--(-1.85,0)--cycle;
  \fill[pattern=north east lines,pattern color=black!55]
       (-2,.25) rectangle (2,1.02);
 \end{scope}
 \draw (-1.01,1.04)--(1.01,1.04);
 \node[fill=white,inner sep=1pt,font=\scriptsize] at (0,1.08)
 {$[n_1]\setminus B$};
 \node[fill=white,inner sep=1pt,font=\scriptsize] at (0,.55)
 {$\cB_y$};
 \node[align=center] at (0,-.9) {$\cB_y\subseteq\cK$\\excluded from $\cF_{y^c}$};
\end{scope}
\draw[->] (1.05,-1.28) .. controls (3,-1.55) and (4.6,.98) .. (6.1,1.04);
\node[fill=white,align=center,font=\scriptsize,inner sep=2pt] at (3.52,-.21)
 {complement in $[n_1]$\\$|[n_1]\setminus B|=r+t+j$};
\node[align=center] at (0,-3) {$A\cup y'\in\cF$};
\node[align=center] at (7.2,-3) {$([n_1]\setminus B)\cup y^c\notin\cF$};
\end{tikzpicture}
\caption{The shadow inclusion \eqref{eq:first-shadow-inclusion} used
in \eqref{eq:exclusion}. Here $y\subseteq y'$, $|y'\setminus y|\le i$,
$|A|=r+t-1-j-i$, and $|A\setminus B|=t-1-i$.
The two displayed subsets of $[n]$ intersect in at most $t-1$ elements,
so $[n_1]\setminus B\in\cB_y^{(r+t+j)}$.
The complete levels $r+t,\ldots,r+t+j-1$ of $\cB_y$ are hatched.
When $i=t-1$, no elements are deleted and $B=A$.}
\label{fig:shadow-complements}
\end{figure}

\begin{claim}\label{claim:thefirstlayer}
\begin{equation}
\begin{aligned}\label{eq:first-layer}
 \mu^{n_1}(\cA_y^{(r+t-1)})-\mu^{n_1}(\cB_y^{(r+t)})\le\frac{4(t-1)}{n_1}
       \left(\mu^{n_1}(\cA_y^{(r+t-1)})-\frac{\theta(n_1,t)}2\right)_+
\end{aligned}
\end{equation}
and
\begin{equation}\label{eq:level-excess-decomposition}
 \phi(y)=
 \left(\mu^{n_1}(\cA_y^{(r+t-1)})-\frac{\theta(n_1,t)}2\right)_+
 +\mu^{n_1}(\cA_y\setminus\cA_y^{(r+t-1)}).
\end{equation}
\end{claim}
\begin{proof}[Proof of the claim~\ref{claim:thefirstlayer}]
The case $i=j=0$ of \eqref{eq:first-shadow-inclusion} gives
$\mu^{n_1}(\cB_y^{(r+t)})\ge
\mu^{n_1}(\partial^{t-1}\cA_y^{(r+t-1)})$.
Among the $(r+t-1)$-subsets in simplicial order, those omitting $1$
follow those containing $1$, and
\begin{equation}\label{eq:half-influence-layer}
 2^{-n_1}\binom{n_1-1}{r+t-1}
 =2^{-n_1}\binom{n_1-1}r=\frac{\theta(n_1,t)}2.
\end{equation}
If $\mu^{n_1}(\cA_y^{(r+t-1)})\le\theta(n_1,t)/2$, every member of
$\cA_y^{(r+t-1)}$ omits $1$.
The normalized shadow bound \eqref{eq:normalized-shadow}, applied
on $[n_1]\setminus\{1\}$, therefore we have
\begin{equation}\label{eq:first-shadow-small}
 \mu^{n_1}(\cB_y^{(r+t)})\ge \mu^{n_1}(\partial^{t-1}\cA_y^{(r+t-1)})
 \ge\frac{\binom{n_1-1}r}{\binom{n_1-1}{r+t-1}}
       \mu^{n_1}(\cA_y^{(r+t-1)})
 =\mu^{n_1}(\cA_y^{(r+t-1)}).
\end{equation}

If $\mu^{n_1}(\cA_y^{(r+t-1)})\ge\theta(n_1,t)/2$, this level contains
all $(r+t-1)$-subsets omitting $1$.
Its shadow therefore contains all $r$-subsets omitting $1$.
Apply \eqref{eq:normalized-shadow} to the restricted projection
$\cA_y^{(r+t-1)}(1)$ on $[n_1]\setminus\{1\}$, and then adjoin $1$
to each set in its $(t-1)$st lower shadow.
These are further $r$-subsets in the shadow of
$\cA_y^{(r+t-1)}$, disjoint from those already counted. Hence
\begin{equation}\label{eq:first-shadow-large}
 \begin{aligned}
 \mu^{n_1}(\partial^{t-1}\cA_y^{(r+t-1)})
 &\ge\frac{\theta(n_1,t)}2+
  \frac{\binom{n_1-1}{r-1}}{\binom{n_1-1}{r+t-2}}
       \left(\mu^{n_1}(\cA_y^{(r+t-1)})-\frac{\theta(n_1,t)}2\right)\\
 &=\frac{\theta(n_1,t)}2+
  \frac{r(r+1)}{(r+t-1)(r+t)}
       \left(\mu^{n_1}(\cA_y^{(r+t-1)})-\frac{\theta(n_1,t)}2\right).
 \end{aligned}
\end{equation}
Since $t\ge2$ and $n_1=2r+t$,
$$
 1-\frac{r(r+1)}{(r+t-1)(r+t)}
 =\frac{n_1(t-1)}{(r+t-1)(r+t)}\le\frac{4(t-1)}{n_1}.
$$
Combining the two cases with \eqref{eq:first-shadow-inclusion} yields
\[
\begin{aligned}
 \mu^{n_1}(\cA_y^{(r+t-1)})-\mu^{n_1}(\cB_y^{(r+t)})&\leq \mu^{n_1}(\cA_y^{(r+t-1)})-\mu^{n_1}(\partial^{t-1}\cA_y^{(r+t-1)})\\
&\le\frac{4(t-1)}{n_1}
       \left(\mu^{n_1}(\cA_y^{(r+t-1)})-\frac{\theta(n_1,t)}2\right)_+.
\end{aligned}
\]
Also $q_0>\theta(n_1,t)/2$, so \eqref{eq:layer-event-meaning} implies \eqref{eq:level-excess-decomposition}.
\end{proof}

For $0\le x<1$, define increasing families on $2^{Y\cup\{z\}}$ by
\begin{equation}\label{eq:layer-events}
 \begin{aligned}
 \mathcal E_j(x)&=\{y:\mu^{n_1}(\cA_y^{(r+t-1-j)})>xq_j\},
       &&1\le j\le r+t-1,\\
 \mathcal M_j(x)&=\{y:\mu^{n_1}(\cB_y^{(r+t+j)})>xq_{j+t-1}\},
       &&1\le j\le r.
 \end{aligned}
\end{equation}
Outside these ranges the families are empty.
Integrating the indicators of these events and summing over the
disjoint levels gives
\begin{equation}\label{eq:layer-integrals}
 \begin{aligned}
 \E_y\mu^{n_1}(\cA_y\setminus\cA_y^{(r+t-1)})
 &=\int_0^1\sum_{j\ge1}q_j\mu^{n-n_1}(\mathcal E_j(x))\,dx,\\
 \E_y\mu^{n_1}(\cB_y\setminus\cB_y^{(r+t)})
 &=\int_0^1\sum_{j=1}^r q_{j+t-1}
                         \mu^{n-n_1}(\mathcal M_j(x))\,dx.
 \end{aligned}
\end{equation}
Indeed, for each $y$ and $q_j>0$,
$q_j\int_0^1\mathbbm{1}_{\mathcal E_j(x)}(y)\,dx
=\mu^{n_1}(\cA_y^{(r+t-1-j)})$; the second identity is identical.

\begin{claim}\label{claim:exclusion}
For $1\le j\le r$, $0\le i\le t-1$ and $0\le x<1$,
\begin{equation}\label{eq:exclusion}
 N_i(\mathcal E_{j+i}(x))\subseteq\mathcal M_j(x).
\end{equation}
\end{claim}
\begin{proof}[Proof of the claim~\ref{claim:exclusion}]
Let $y\in N_i(\mathcal E_{j+i}(x))$.
Since $\mathcal E_{j+i}(x)$ is increasing, there is
$y'\supseteq y$ with $|y'\setminus y|\le i$ and
$y'\in\mathcal E_{j+i}(x)$.
By \eqref{eq:first-shadow-inclusion} and
\eqref{eq:normalized-shadow},
$$
 \frac{\mu^{n_1}(\cB_y^{(r+t+j)})}{q_{j+t-1}}
 \ge\frac{|\partial^{t-1-i}\cA_{y'}^{(r+t-1-j-i)}|}
             {\binom {n_1}{r-j}}
 \ge\frac{\mu^{n_1}(\cA_{y'}^{(r+t-1-j-i)})}{q_{j+i}}>x.
$$
Thus $y\in\mathcal M_j(x)$.
\end{proof}

\begin{claim}\label{remaininglevels}
\begin{equation}\label{eq:deeper-target}
\E_y\mu^{n_1}(\cA_y\setminus\cA_y^{(r+t-1)})-\E_y\mu^{n_1}(\cB_y\setminus\cB_y^{(r+t)})\le\frac{8(t-1)}{n_1\zeta}\E_y\mu^{n_1}(\cA_y\setminus\cA_y^{(r+t-1)}).
\end{equation}
\end{claim}

\begin{proof}[Proof of Claim \ref{remaininglevels}]

For $j\ge1$, membership in $\mathcal E_j(x)$ forces the first
level of $\cA_y$ to be complete. Hence
$\mathcal E_j(x)\subseteq\{\phi>0\}$ and
$\mu^{n-n_1}(\mathcal E_j(x))\le\rho$ by
\eqref{eq:projection-rare-support}.
Combining \eqref{eq:multiradius-hypothesis} and
\eqref{eq:exclusion}, including $i=0$, gives
\begin{equation}\label{eq:layer-neighborhood-measures}
 \mu^{n-n_1}(\mathcal M_j(x))
 \ge\max_{0\le i\le t-1}(1+\zeta)^i
               \mu^{n-n_1}(\mathcal E_{j+i}(x)).
\end{equation}
We take a maximum, not a sum: neighborhoods for different $i$
may overlap. The sum over $j$ in \eqref{eq:layer-integrals} is
valid because the corresponding levels of $\cB_y$ are disjoint.

The ratios $q_{j+1}/q_j=(r+t-1-j)/(r+j+2)$ decrease with $j$,
so $q_1,\ldots,q_{r+t-1}$ are log-concave.
Apply Lemma~\ref{lem:interval-coupling} to this sequence, reindexed
from zero, with $h=t-1$, $\gamma=1+\zeta$, and
$z_{j-1}=\mu^{n-n_1}(\mathcal E_j(x))$.
For each fixed $x$ it gives
$$
\sum_{j\ge1}q_{j+t-1}\max_{0\le i\le t-1}(1+\zeta)^i\mu^{n-n_1}(\mathcal E_{j+i}(x))\ge\frac{\sum_{j\ge1}q_{j+t-1}(1+\zeta)^{-j}}{\sum_{j\ge1}q_j(1+\zeta)^{-j}}
\sum_{j\ge1}q_j\mu^{n-n_1}(\mathcal E_j(x)).
$$
The common factor from the reindexing cancels in the ratio.
Integrating and using \eqref{eq:layer-integrals} and
\eqref{eq:layer-neighborhood-measures}, we obtain
\begin{equation}\label{eq:deep-loss}
\E_y\mu^{n_1}(\cA_y\setminus\cA_y^{(r+t-1)})-\E_y\mu^{n_1}(\cB_y\setminus\cB_y^{(r+t)})\le\frac{\sum_{j\ge1}(q_j-q_{j+t-1})(1+\zeta)^{-j}}{\sum_{j\ge1}q_j(1+\zeta)^{-j}}\E_y\mu^{n_1}(\cA_y\setminus\cA_y^{(r+t-1)}).
\end{equation}

For $1\le j\le r$,
\begin{equation}\label{eq:shadow-layer-ratio}
 \begin{aligned}
 \frac{q_{j+t-1}}{q_j}
 &=\frac{\binom {n_1}{r-j}}{\binom {n_1}{r+t-1-j}}
 =\prod_{a=1}^{t-1}\frac{r+a-j}{r+a+j+1},\\
 1-\frac{q_{j+t-1}}{q_j}
 &\le\sum_{a=1}^{t-1}\frac{2j+1}{r+a+j+1}
 \le\frac{6(t-1)(j+1)}{n_1}.
 \end{aligned}
\end{equation}
Here each factor is in $[0,1]$, and $r=(n_1-t)/2\ge3n_1/8$.
For $j>r$ with $q_j>0$, we have $q_{j+t-1}=0$ and
$6(t-1)(j+1)/n_1\ge1$, so the final bound still holds.
Also,
$$
 \frac{q_{j+1}}{q_j}
 =1+\frac{t-3-2j}{r+j+2}
 \le1+\frac{3(t-1)}{n_1}\le1+\frac\zeta8.
$$
Consequently, for each integer $s\ge1$,
$$
 \sum_{j\ge s+1}q_j(1+\zeta)^{-j}
 \le\left(\frac{1+\zeta/8}{1+\zeta}\right)^s
       \sum_{j\ge1}q_j(1+\zeta)^{-j}.
$$
Using $j+1=2+\sum_{s\ge1}\mathbbm{1}_{\{j\ge s+1\}}$ gives
$$
\frac{\sum_{j\ge1}(j+1)q_j(1+\zeta)^{-j}}{\sum_{j\ge1}q_j(1+\zeta)^{-j}}\le2+\sum_{s\ge1}\left(\frac{1+\zeta/8}{1+\zeta}\right)^s=1+\frac{8(1+\zeta)}{7\zeta}\le\frac4{3\zeta},
$$
where the last step uses $\zeta\le1/16<4/45$.
Together with \eqref{eq:shadow-layer-ratio}, this proves
\eqref{eq:deeper-target}.
\end{proof}

Finally, add \eqref{eq:first-layer} and
\eqref{eq:deeper-target}, averaging the former over $y$.
By \eqref{eq:projection-mean-balance} and
\eqref{eq:level-excess-decomposition},
$$
\begin{aligned}
\mu^n(\cF)-w(n_1,t)&\le\frac{4(t-1)}{n_1}\E_y\left(\mu^{n_1}(\cA_y^{(r+t-1)})-\frac{\theta(n_1,t)}2\right)_+ +\frac{8(t-1)}{n_1\zeta}\E_y\mu^{n_1}(\cA_y\setminus\cA_y^{(r+t-1)})\\
&\le\frac{8(t-1)}{n_1\zeta}\E_y\phi.
 \end{aligned}
$$
The right-hand side is nonnegative, so this also bounds
$(\mu^n(\cF)-w(n_1,t))_+$.
\end{proof}

\begin{remark*}[Why the positive radii are used]
For a fixed $x$, the case $i=0$ of \eqref{eq:exclusion} gives only
$\mu^{n-n_1}(\mathcal M_j(x))\ge\mu^{n-n_1}(\mathcal E_j(x))$.
Consequently it gives
$$
\sum_{j\ge1}q_j\mu^{n-n_1}(\mathcal E_j(x))-\sum_{j=1}^r q_{j+t-1}\mu^{n-n_1}(\mathcal M_j(x))\le\sum_{j\ge1}(q_j-q_{j+t-1})\mu^{n-n_1}(\mathcal E_j(x)).
$$
This bound need not have the coefficient $8(t-1)/(n_1\zeta)$
required in \eqref{eq:deeper-target}. We give an explicit calculation.
Take $t=2$, $n_1=2^{16}$ and $\zeta=1/16$, so $r=n_1/2-1$.
For the numerical inequalities above, assign the same positive
value $a\le\rho/2$ to every $\mu^{n-n_1}(\mathcal E_j(x))$,
$1\le j\le r+1$, and take
$\mu^{n-n_1}(\mathcal M_j(x))=a$ for $1\le j\le r$.
These values satisfy all the $i=0$ estimates. They are test values
for those estimates, not a construction of a $t$-intersecting family.
Binomial symmetry gives
$$
 q_0=2^{-n_1}\binom {n_1}{n_1/2},\qquad
 q_1=\frac {n_1}{n_1+2}q_0,\qquad
 \sum_{j=1}^{r+1}q_j=\frac{1-q_0}{2}.
$$
The added measure is $a\sum_{j=1}^{r+1}q_j$, whereas the lower
bound for the missing measure is
$a\sum_{j=1}^r q_{j+1}=a(\sum_{j=1}^{r+1}q_j-q_1)$.
Their relative difference is therefore
$$
 \frac{q_1}{\sum_{j=1}^{r+1}q_j}
 =\frac{2n_1 q_0}{(n_1+2)(1-q_0)}
 >q_0\ge\frac1{512}
 =\frac{8(t-1)}{n_1\zeta}.
$$
Here $q_0=\theta(n_1,2)$ and the last inequality follows from
\eqref{eq:centralw}. Thus the zero-radius estimates alone do not
imply the required bound, even for constant values of the level measures.

Keeping $\mu^{n-n_1}(\mathcal E_j(x))=a$, the $i=1$ estimate requires
$\mu^{n-n_1}(\mathcal M_j(x))\ge(17/16)a$ for every $1\le j\le r$.
By \eqref{eq:profile-upper-sharp}, $q_0\le1/320$, and hence
$$
 \frac{\sum_{j=1}^r q_{j+1}\mu^{n-n_1}(\mathcal M_j(x))}
      {\sum_{j=1}^{r+1}q_j\mu^{n-n_1}(\mathcal E_j(x))}
 \ge\frac{17}{16}\left(1-
               \frac{q_1}{\sum_{j=1}^{r+1}q_j}\right)
 \ge\frac{17}{16}\left(1-\frac2{319}\right)>1.
$$
Thus radius one already makes the missing measure exceed the added
measure in this calculation.
For arbitrary values of the level measures, all the positive-radius
bounds in \eqref{eq:layer-neighborhood-measures} allow
Lemma~\ref{lem:interval-coupling} to replace the unweighted
coefficient above by the one in \eqref{eq:deep-loss}.
The calculation in the proof gives
$$
 \frac{\sum_{j\ge1}(q_j-q_{j+t-1})(1+\zeta)^{-j}}
      {\sum_{j\ge1}q_j(1+\zeta)^{-j}}
 \le\frac{6(t-1)}{n_1}\,
      \frac{\sum_{j\ge1}(j+1)q_j(1+\zeta)^{-j}}
           {\sum_{j\ge1}q_j(1+\zeta)^{-j}}
 \le\frac{8(t-1)}{n_1\zeta}.
$$
This is the coefficient in \eqref{eq:deeper-target}.
\end{remark*}

\subsection{Hamming neighborhoods of small families}\label{subsec:rare-expansion}

\begin{lemma}\label{lem:rare}
Suppose $m\ge8192$, $h\ge1$, $32\ln(m+1)\le L\le m/64$, and
$mL\ge1024h^2$. Every family $\cF\subseteq2^{[m]}$ with
$\mu(\cF)\le e^{-L}$ satisfies, for $1\le i\le h$,
\begin{equation}\label{eq:radius-expansion}
 \mu(N_i(\cF))\ge
       \left(1+\frac{\sqrt L}{2\sqrt m}\right)^i\mu(\cF).
\end{equation}
\end{lemma}
\begin{proof}
Set $r_0=\lfloor\sqrt{mL}/4\rfloor$ and $k=\lfloor m/2\rfloor-r_0$. Since $\sqrt{mL}\ge32h\ge32$ and $L\le m/64$, we have $r_0\ge7\sqrt{mL}/32\ge4h$ and
$r_0+1\le9\sqrt{mL}/32\le m/8$.

We first give a lower bound for $\binom mk$. The coefficient $\binom m{\lfloor m/2\rfloor}$ is the largest of the $m+1$ binomial coefficients whose sum is $2^m$, so $2^{-m}\binom m{\lfloor m/2\rfloor}\ge1/(m+1)$.
For every integer $\ell$ with $k+1\le\ell\le\lfloor m/2\rfloor$,
we have $0\le m-2\ell+1\le2(r_0+1)$ and $m-\ell+1\ge m/2$. Therefore
\begin{equation*}
\frac{\binom m{\ell-1}}{\binom m\ell}
=\frac{\ell}{m-\ell+1}
=1-\frac{m-2\ell+1}{m-\ell+1}
\ge1-\frac{4(r_0+1)}m
\ge\frac12.
\end{equation*}
Using $\ln(1-v)\ge-2v$ for $0\le v\le1/2$, together with
$\ln(m+1)\le L/32$ and $8(r_0+1)^2/m\le81L/128$, gives
\begin{equation}\label{eq:rare-layer-mass}
\begin{aligned}
2^{-m}\binom mk &=2^{-m}\binom m{\lfloor m/2\rfloor}\prod_{\ell=k+1}^{\lfloor m/2\rfloor}\frac{\binom m{\ell-1}}{\binom m\ell}\ge\frac1{m+1}\left(1-\frac{4(r_0+1)}m\right)^{r_0}\\
&\ge\frac1{m+1}\exp\left(-\frac{8(r_0+1)^2}{m}\right)\ge \mathrm e^{-85L/128}\ge \mathrm e^{-2L/3}>\mathrm e^{-L}.
\end{aligned}
\end{equation}

Let $\cF_0$ be the initial segment of the simplicial order
with $|\cF_0|=|\cF|$. By \eqref{eq:rare-layer-mass}, $|\cF_0|\le2^m\mathrm e^{-L}<\binom mk\le\sum_{j=0}^k\binom mj$. Since every set of size at most $k$ precedes every set of larger size in the simplicial order, $\cF_0=\bigcup_{j=0}^k\cF_0^{(j)}$.

Fix an integer $i$ with $1\le i\le h$. Since $i\le r_0/4$, we have
$k+i\le m/2-3r_0/4<m$. All complements below are taken in $[m]$. For each $0\le j\le k$, count the pairs $(A,B)$ with $A\in(\cF_0^{(j)})^c$, $B\subseteq A$, and $|B|=m-j-i$. Each $A\in(\cF_0^{(j)})^c$ contains exactly $\binom{m-j}{i}$ such sets $B$. The sets $B$ that occur form the lower shadow $\partial^i((\cF_0^{(j)})^c)$, and each is contained in at most
$\binom{j+i}{i}$ sets of size $m-j$. Consequently,
\begin{equation*}
|\cF_0^{(j)}|\binom{m-j}{i}\le\bigl|\partial^i((\cF_0^{(j)})^c)\bigr|\binom{j+i}{i}.
\end{equation*}

The bound $k+i\le m/2-3r_0/4$ gives
\begin{equation*}
 \frac{m-k-i+1}{k+i}
 =1+\frac{m-2(k+i)+1}{k+i}
 \ge1+\frac{3r_0}{m}
 \ge1+\frac{\sqrt L}{2\sqrt m},
\end{equation*}
where the last inequality uses $r_0\ge7\sqrt{mL}/32$.
For every $0\le j\le k$, we therefore have
\begin{equation}\label{eq:upper-shadow-expansion}
\frac{\binom{m-j}{i}}{\binom{j+i}{i}}=\prod_{\ell=0}^{i-1}\frac{m-j-\ell}{j+i-\ell}\ge\left(\frac{m-k-i+1}{k+i}\right)^i\ge\left(1+\frac{\sqrt L}{2\sqrt m}\right)^i.
\end{equation}
Combining the pair count with~\eqref{eq:upper-shadow-expansion} yields
\begin{equation*}
\bigl|\partial^i((\cF_0^{(j)})^c)\bigr|\ge\left(1+\frac{\sqrt L}{2\sqrt m}\right)^i|\cF_0^{(j)}|.
\end{equation*}

For every $B\in\partial^i((\cF_0^{(j)})^c)$, there is
$A\in\cF_0^{(j)}$ such that $B\subseteq A^c$.
Thus $A\subseteq B^c$ and $|B^c\setminus A|=i$, which proves
\begin{equation*}
 \bigl(\partial^i((\cF_0^{(j)})^c)\bigr)^c
 \subseteq N_i(\cF_0)\cap\binom{[m]}{j+i}.
\end{equation*}
These families are pairwise disjoint as $j$ varies, because
their members have distinct sizes $j+i$. Since taking complements preserves cardinality, \eqref{eq:classical-harper} now gives
\begin{equation*}
|N_i(\cF)|\ge |N_i(\cF_0)|\ge \sum_{j=0}^k\bigl|\partial^i((\cF_0^{(j)})^c)\bigr|\ge
\left(1+\frac{\sqrt L}{2\sqrt m}\right)^i\sum_{j=0}^k|\cF_0^{(j)}|=\left(1+\frac{\sqrt L}{2\sqrt m}\right)^i|\cF|.
\end{equation*}
Dividing by $2^m$ proves~\eqref{eq:radius-expansion}.
\end{proof}

\begin{remark*}
Lemma~\ref{lem:rare} is a quantitative consequence of Harper's vertex-isoperimetric theorem. The proof gives the binomial estimates that give the factor in \eqref{eq:radius-expansion} under the numerical
hypotheses used here.
\end{remark*}

\subsection{A lower bound for the covariance}
\label{subsec:excess-covariance}

Only the smoothing estimate from Section~\ref{subsec:smoothing} and
edge isoperimetry are needed for the comparison below.

\begin{lemma}\label{lem:excess-covariance}
For the families and coordinate sets in Proposition~\ref{prop:localization},
the functions $\phi_\cA,\phi_\cB$ defined in \eqref{eq:excess-functions} satisfy
\begin{equation}\label{eq:rare-correlation}
\Cov(f,g^*)\ge\frac1{12}\sqrt{\frac{L-\ln2}{n_2}}\,\E\phi_\cA,
\qquad
\Cov(f,g^*)\ge\frac1{12}\sqrt{\frac{L-\ln2}{n_1}}\,\E\phi_\cB.
\end{equation}
\end{lemma}
\begin{proof}
We first prove a bound for an arbitrary increasing Boolean function
$u$ depending only on $Y\cup\{z\}$, with
$0<\alpha:=\E u\le e^{-L}$. Set $a=\ln(1/(2\alpha))$.
For $b\in\{0,1\}$, let $u_b(y)=u(y,b)$ for $y\in\{0,1\}^Y$, and set $\alpha_b=2^{-n_2}\sum_yu_b(y)$.
These are the two restrictions obtained by fixing coordinate $z$.
We have $\alpha=(\alpha_0+\alpha_1)/2$ and
$\alpha_b\le2\alpha$.
Also $I_Y[u]=(I[u_0]+I[u_1])/2$, since $Y$ contains all the
remaining coordinates on which $u$ depends.
Applying \eqref{eq:edge-isoperimetry} on $\{0,1\}^Y$ gives
\begin{equation}\label{eq:rare-section-influence}
I_Y[u]\ge\frac1{\ln2}\left(\alpha_0\ln\frac1{\alpha_0}+\alpha_1\ln\frac1{\alpha_1}\right)\ge\frac{\alpha_0+\alpha_1}{\ln2}\ln\frac1{2\alpha}=\frac{2\alpha a}{\ln2}.
\end{equation}
Terms with $\alpha_b=0$ are interpreted as zero.

To apply Lemma~\ref{lem:smoothed-threshold}, use the explicit function
$$
 h=(\mathbbm{1}_{\cK(n,Y,q)})^*,
 \qquad
 h(x)=\mathbbm{1}_{\{|x_Y|>\frac{n_2-q}{2}\}}.
$$
Here $|x_Y|=\sum_{i\in Y}x_i$.
The complement map preserves uniform measure, so
$\E|g^*-h|=\delta_g\le1/(64n)$ by
\eqref{eq:localization}.
By \eqref{eq:block-ratio} and the assumed bound on $n$,
$n_2\ge2^{14}(p+q-2)^{7/3}>256q^2+2$, and $n\ge2^{18}$.
Also, using \eqref{eq:centralw},
$$
 (2\delta_g)^{4/5}
 \le\frac1{16n^{4/5}}
 \le\frac1{8\sqrt {n_2}}\le\frac{\theta(n_2,q)}4.
$$
Thus all the dimension and error hypotheses of Lemma~\ref{lem:smoothed-threshold} hold,
and its constant $24$ is available. Combining it with \eqref{eq:covariance-integral} and
retaining only coordinates in $Y$ gives
\begin{equation}\label{eq:rare-indicator-correlation}
 \begin{aligned}
\Cov(u,g^*)&\ge\frac1{24\sqrt {n_2}}\sum_{i\in Y}\frac{I_i[u]}{\sqrt{1+\ln(1/I_i[u])}}\ge\frac{I_Y[u]}{24\sqrt {n_2}\sqrt{1+\ln(n_2/I_Y[u])}}\\
&\ge\frac{2\alpha a/\ln2}{24\sqrt {n_2}\sqrt{1+\ln(n_2\ln2/(2\alpha a))}}.
 \end{aligned}
\end{equation}
All omitted terms in the covariance integral are nonnegative since $u,g$ are both increasing.
For the second inequality in \eqref{eq:rare-indicator-correlation}, apply Jensen's inequality to the convex
function $s\mapsto \frac{s}{\sqrt{1+\ln(1/s)}}$ on $[0,1]$, with value
zero at $s=0$, and use $|Y|=n_2$.
The last inequality follows from \eqref{eq:rare-section-influence},
since $s\mapsto \frac{s}{\sqrt{1+\ln(n_2/s)}}$ is increasing on $(0,n_2]$.

By \eqref{eq:finalnumbers},
$a\ge L-\ln2\ge1+\ln n_2\ge1$.
As $2\alpha=e^{-a}$, we obtain
$$
 1+\ln\frac{n_2\ln2}{2\alpha a}
 =a+1+\ln n_2+\ln(\ln2)-\ln a\le2a.
$$
It follows from \eqref{eq:rare-indicator-correlation} and
$\ln2<1/\sqrt2$ that
\begin{equation}\label{eq:rare-indicator-simple}
 \Cov(u,g^*)\ge
 \frac{\sqrt2}{24\ln2}\sqrt{\frac a{n_2}}\,\alpha
 \ge\frac1{12}\sqrt{\frac{L-\ln2}{n_2}}\,\E u.
\end{equation}
This is also true when $u=0$.

For each fixed $y\subseteq Y\cup\{z\}$, the restrictions of
$f(\cdot,y)$ and $g^*(\cdot,y)$ to $2^X$ are increasing.
Their covariance is nonnegative by Harris--Kleitman.
The average of $f(\cdot,y)$ on this projection is $\mu^{n_1}(\cA_y)$; the average of $g^*(\cdot,y)$ is
$1-\mu^{n_1}(\cB(y^c,Y\cup\{z\}))$, where
$y^c=(Y\cup\{z\})\setminus y$.
Thus \eqref{eq:total-covariance} gives
\begin{equation}\label{ineq:lower-bound-cov}
\Cov(f,g^*)\ge\Cov_y\bigl(\mu^{n_1}(\cA_y),1-\mu^{n_1}(\cB(y^c,Y\cup\{z\}))\bigr).
\end{equation}
Here $\Cov_y$ is taken over the uniform $y\subseteq Y\cup\{z\}$.

We next show that replacing $\mu^{n_1}(\cA_y)$ by $\phi_\cA(y)$ cannot increase the covariance on the RHS of~\eqref{ineq:lower-bound-cov}. Since $\cA$ and $\cB$ are increasing, for $y\subseteq y'$ we have $\cA_y\subseteq\cA_{y'}$ and $\cB((y')^c,Y\cup\{z\}) \subseteq\cB(y^c,Y\cup\{z\})$. Consequently, both $\mu^{n_1}(\cA_y)$ and $1-\mu^{n_1}(\cB(y^c,Y\cup\{z\}))$ are increasing in $y$. Moreover, by the definition of $\phi_\cA$,
$$
\mu^{n_1}(\cA_y)-\phi_\cA(y)=\min\left\{\mu^{n_1}(\cA_y),
w(n_1,p)+\frac{\theta(n_1,p)}2\right\}.
$$
Taking the minimum with a fixed constant preserves increasing property, so this difference is also increasing in $y$. Another application of Harris--Kleitman yields
$$
\Cov_y\bigl(\mu^{n_1}(\cA_y)-\phi_\cA(y),1-\mu^{n_1}(\cB(y^c,Y\cup\{z\}))\bigr)\ge0.
$$
By linearity of covariance, we obtain
\begin{equation}\label{eq:conditional-clipping}
\Cov(f,g^*)\ge\Cov_y\bigl(\phi_\cA(y),1-\mu^{n_1}(\cB(y^c,Y\cup\{z\}))\bigr)=\Cov(\phi_\cA,g^*).
\end{equation}
The last equality follows by averaging over $X$, since $\phi_\cA$
depends only on $y$.

For $s\ge0$, the function
$u_s=\mathbbm{1}_{\{\phi_\cA>s\}}$ is increasing and depends only
on $Y\cup\{z\}$. By \eqref{eq:excess-support},
$\E u_s\le\mu^{n_2+1}(\phi_\cA>0)\le e^{-L}$.
Thus \eqref{eq:rare-indicator-simple} applies to every $u_s$.
Since $0\le\phi_\cA\le1$, the point-wise identity
$\phi_\cA=\int_0^1\mathbbm{1}_{\{\phi_\cA>s\}}\,ds$ gives
$$
\Cov(\phi_\cA,g^*)=\int_0^1\Cov(u_s,g^*)\,ds\ge\frac1{12}\sqrt{\frac{L-\ln2}{n_2}}\int_0^1\E u_s\,ds
 =\frac1{12}\sqrt{\frac{L-\ln2}{n_2}}\,\E\phi_\cA.
$$
Together with \eqref{eq:conditional-clipping}, this proves the first
of the two bounds
$$
\Cov(f,g^*)\ge\frac1{12}\sqrt{\frac{L-\ln2}{n_2}}\,\E\phi_\cA,
\qquad
\Cov(f,g^*)\ge\frac1{12}\sqrt{\frac{L-\ln2}{n_1}}\,\E\phi_\cB.
$$
For the second, exchange $(f,X,p)$ and $(g,Y,q)$ and use
$\Cov(g,f^*)=\Cov(f,g^*)$, which follows by the change of
variables $x\mapsto1-x$ in the uniform expectation.
\end{proof}

\subsection{Completion of the proof}

\begin{proof}[Proof of Theorem~\ref{thm:main}]
Note that the product construction implies $w(n,p,q)\geq B_n$.
By Fact \ref{fact-reduction}, it suffices to show $w(n,p,q)\leq B_n$ when $n-p-q$ is odd.
Suppose that there is a $(p,q)$-family $\cF\subset 2^{[n]}$ with $\mu(\cF)>B_n$.  Apply Lemma~\ref{lem:closure} and
Proposition~\ref{prop:localization}, and retain the resulting
$f,g,X,Y,z,n_1,n_2,L$ and the functions $\phi_\cA,\phi_\cB$ in \eqref{eq:excess-functions}.

Apply Lemma~\ref{lem:rare} to families in $2^{Y\cup\{z\}}$.
This cube has $n_2+1$ coordinates. By \eqref{eq:block-ratio} and
Lemma~\ref{lem:numerical},
$$
 n_2+1\ge8192,\qquad
 32\ln(n_2+2)\le L\le\frac{n_2+1}{64},\qquad
 (n_2+1)L\ge1024(p-1)^2.
$$
Thus the lemma applies with $m=n_2+1$ and $h=p-1$. Set $\zeta=\frac{\sqrt L}{2\sqrt{n_2+1}}$. We have $\zeta\le1/16$ because $L\le(n_2+1)/64$.
Also \eqref{eq:geometric} and \eqref{eq:finalnumbers} imply
$$
 \frac{48(p-1)\sqrt{n_2+1}}{n_1}
 \le\frac{54(p+q-2)}{\sqrt n}\le\sqrt L,
$$
so $\zeta\ge24(p-1)/n_1$.
Moreover, \eqref{eq:localization} gives
$$
 \frac{2\mu(\cA\triangle\cK(n,X,p))}{\theta(n_1,p)}
 =\frac{2\delta_f}{\theta(n_1,p)}\le e^{-L}.
$$
For every increasing $\cG\subseteq2^{Y\cup\{z\}}$ with
$\mu(\cG)\le e^{-L}$, Lemma~\ref{lem:rare} gives
$\mu(N_i(\cG))\ge(1+\zeta)^i\mu(\cG)$ for each $1\le i\le p-1$. Finally, \eqref{eq:block-ratio} gives $n_1\ge 4t$. Thus all hypotheses of Theorem~\ref{thm:exterior} hold with $t=p$ and $\rho=e^{-L}$. Hence
\begin{equation}\label{eq:final-density-gain}
 (\mu(\cA)-w(n_1,p))_+
 \le\frac{8(p-1)}{n_1\zeta}\,\E\phi_\cA
\overset{\eqref{eq:rare-correlation}}{\le }\frac{96(p-1)}{n_1\zeta}\, \sqrt{\frac{n_2}{L-\ln2}} \Cov(f,g^*).
\end{equation}
Note that
\begin{equation}\label{eq:final-comparison}
 \begin{aligned}
 \frac{96(p-1)}{n_1\zeta}\, \sqrt{\frac{n_2}{L-\ln2}}
 =\frac{192(p-1)\sqrt{n_2(n_2+1)}}{n_1\sqrt{L(L-\ln2)}}\le\frac{170(p+q-2)}L<\frac23,
 \end{aligned}
\end{equation}
where for the first inequality, we use \eqref{eq:geometric} and
$L-\ln2\ge63L/64$; then
$192(7/8)\sqrt{64/63}<170$.
For the second, use $L>256(p+q-2)$ and $170/256<2/3$.
Thus $(\mu(\cA)-w(n_1,p))_+ \leq \frac23\Cov(f,g^*)$.  Similarly, $(\mu(\cB)-w(n_2,q))_+ \leq \frac23\Cov(f,g^*)$.

As $\mu(\cA),\mu(\cB),w(n_1,p),w(n_2,q)\le1/2$, we conclude that
\begin{equation}\label{eq:final-product-gain}
 \begin{aligned}
 \mu(\cA)\,\mu(\cB)-w(n_1,p)w(n_2,q)
 &\le\frac12(\mu(\cA)-w(n_1,p))_+
       +\frac12(\mu(\cB)-w(n_2,q))_+\le\frac23\Cov(f,g^*).
 \end{aligned}
\end{equation}
Together with \eqref{eq:pair}, this implies
$\mu(\cF)\le \mu(\cA\cap\cB^c)\le w(n_1,p)w(n_2,q)\le B_n$, contradicting our assumption.
\end{proof}

\appendix

\section{Deferred proofs of claims}\label{app:claims}

\begin{proof}[Proof of Claim~\ref{cl:product-increment}]
We use induction on $n_1+n_2$.
The case $n_1+n_2=n-\eps$ follows from the definition of $B_n$.

Suppose that $n_1+n_2\ge n-\eps+2$.
By \eqref{eq:incrementa}, \eqref{eq:profile-upper-sharp},
and $w(n_1,p),w(n_2,q)<1/2$,
$$
 \begin{aligned}
 &\bigl(w(n_1,p)-w(n_1-2,p)\bigr)w(n_2,q)\le\frac{p-1}{5(n_1-1)\sqrt {n_1}},\\
 &w(n_1,p)\bigl(w(n_2,q)-w(n_2-2,q)\bigr)
 \le\frac{q-1}{5(n_2-1)\sqrt {n_2}}.
 \end{aligned}
$$
By the convexity of $(x-1)\sqrt x$ and $n_1+n_2\geq n+1$,
\begin{equation}\label{eq:greedy-step}
\begin{aligned}
\frac15\min\left\{\frac{p-1}{(n_1-1)\sqrt {n_1}},\frac{q-1}{(n_2-1)\sqrt {n_2}}\right\}
&\le\frac{p+q-2}{5\bigl((n_1-1)\sqrt {n_1}+(n_2-1)\sqrt {n_2}\bigr)}\\
&\le\frac{\sqrt2(p+q-2)}{5(n_1+n_2-2)\sqrt{n_1+n_2}}\le\frac{\sqrt2(p+q-2)}{5(n-1)\sqrt{n+1}}.
\end{aligned}
\end{equation}
Choose the smaller of the two upper bounds in \eqref{eq:greedy-step} and decrease the corresponding dimension by two. This choice preserves the lower dimension constraints. Indeed, if $n_1-2<256p^2$, then
$n_1\le256p^2+2\le1026(p-1)^2<n/2$ and $n_2>n/2\ge32768(p+q-2)^2$.
In particular, $n_2-2\ge256q^2$, and
$$
 \frac{(q-1)(n_1-1)\sqrt {n_1}}{(p-1)(n_2-1)\sqrt {n_2}}
 \le\frac{2(q-1)}{p-1}\left(\frac {n_1}{n_2}\right)^{3/2}
 \le\frac{(q-1)(p-1)^2}{32(p+q-2)^3}<1.
$$
Thus the smaller bound corresponds to decreasing $n_2$. The other
boundary case is symmetric. The induction hypothesis and
\eqref{eq:greedy-step} now give
$$
 w(n_1,p)w(n_2,q)-B_n
 \le\frac{(p+q-2)(n_1+n_2-n+\eps-2)}{5(n-1)\sqrt{2(n+1)}}
     +\frac{\sqrt2(p+q-2)}{5(n-1)\sqrt{n+1}},
$$
which is \eqref{ineq-3.1}.
\end{proof}

\begin{proof}[Proof of Lemma~\ref{cl:overlap-endpoint}]
If $w(n_1,p)w(n_2,q)\le B_n$, \eqref{eq:overlap-endpoint} is immediate.
Otherwise, by admissibility we may assume $n_1+n_2\ge n-\eps+2$.

For $n+1\le n_1+n_2\le2n$,  we have
\begin{align}\label{ineq-3.2}
 (n_1+n_2)(n_1+n_2-n+\eps)
 &\le(n_1+n_2)(n_1+n_2-n+1)\le2(n+1)(n_1+n_2-n).
\end{align}
By \eqref{eq:profile-lower-central}, \eqref{ineq-3.1} and
$\sqrt{n_1n_2}\le(n_1+n_2)/2$, we obtain
\begin{align*}
 \frac{w(n_1,p)w(n_2,q)-B_n}{\theta(n_1,p)\theta(n_2,q)}
 &\le
 \left(\frac{34}{27}\right)^2
 \frac{(p+q-2)(n_1+n_2)(n_1+n_2-n+\eps)}
      {10(n-1)\sqrt{2(n+1)}}\\
 &\overset{\eqref{ineq-3.2}}{\le}
 \left(\frac{34}{27}\right)^2
 \frac{(p+q-2)\sqrt{n+1}(n_1+n_2-n)}
      {5\sqrt2(n-1)}\\
 &\le
 \frac{9(p+q-2)\sqrt{n+1}(n_1+n_2-n)}{40(n-1)}\\
 &\le
 \frac{5(p+q-2)(n_1+n_2-n)}{22\sqrt n},
\end{align*}
where the third inequality follows from  $\frac1{5\sqrt2}\left(\frac{34}{27}\right)^2<\frac9{40}$ and the last inequality follows from $ \frac{9\sqrt{n+1}}{40(n-1)}
 \le\frac5{22\sqrt n}$ for $n\geq 256$.
Rearranging proves \eqref{eq:overlap-endpoint}.
\end{proof}

\begin{proof}[Proof of Claim~\ref{cl:slice-recurrence}]
For each $x\in \{0,1\}^{m-1}$ with $D_m h(x)=1$, we have
$$
\sum_{\ell=1}^{m-1}\chi_{\{\ell\}}(x)=\sum_{\ell=1}^{m-1} (-1)^{x_\ell} = \sum_{\ell=1}^{m-1}(1-2x_\ell)=-(t-1).
$$
It follows that
\begin{align}\label{ineq-3.3}
-(t-1)D_m h =D_m h\sum_{\ell=1}^{m-1}\chi_{\{\ell\}}.
\end{align}
Since
$$
 \chi_{\{\ell\}}\chi_{[j-1]}
 =
 \begin{cases}
  \chi_{[j-1]\setminus\{\ell\}},&1\le\ell\le j-1,\\
  \chi_{[j-1]\cup\{\ell\}},&j\le\ell\le m-1,
 \end{cases}
$$
taking inner product with $\chi_{[j-1]}$ on both sides of \eqref{ineq-3.3} we obtain
$$
 \begin{aligned}
 -(t-1)\widehat{D_m h}([j-1])
 &=
 \sum_{\ell=1}^{j-1}
 \widehat{D_m h}([j-1]\setminus\{\ell\})+
 \sum_{\ell=j}^{m-1}
 \widehat{D_m h}([j-1]\cup\{\ell\})\\
 &=
 (j-1)\widehat{D_m h}([j-2])
 +(m-j)\widehat{D_m h}([j]).
 \end{aligned}
$$
By rearranging we obtain \eqref{ineq-3.4}.
Moreover, for $j=1$ and $j=2$ we have
\begin{align*}
-(t-1)\widehat{D_m h}(\varnothing)&= (m-1)\widehat{D_m h}(\{1\}),\\[3pt]
-(t-1)\widehat{D_m h}(\{1\})&= \widehat{D_m h}(\varnothing) +(m-2)\widehat{D_m h}([2])=\theta(m,t) +(m-2)\widehat{D_m h}([2]).
\end{align*}
Thus \eqref{ineq-3.6} follows.
\end{proof}

\begin{proof}[Proof of Claim~\ref{cl:slice-symmetry}]
Whenever $D_m h(x)=1$, we have
$$
 \chi_{[m-1]}(x)
 =(-1)^{\sum_{\ell=1}^{m-1}x_\ell}
 =(-1)^{(m+t)/2-1}.
$$
Since
$\chi_{[m-1]\setminus[j]}=\chi_{[m-1]}\chi_{[j]}$,
it follows that
$$
 \begin{aligned}
 \widehat{D_m h}([m-1]\setminus[j])
 &=\E[D_m h\,\chi_{[m-1]}\chi_{[j]}]=(-1)^{(m+t)/2-1}\widehat{D_m h}([j]).
 \end{aligned}
$$
Recall that Fourier coefficients
depend only on the cardinality of the indexing set. Hence
$$
 \widehat{D_m h}([m-1-j])
 =\widehat{D_m h}([m-1]\setminus[j]).
$$
Taking absolute values, we obtain \eqref{ineq-3.7}.
\end{proof}

\section{Proofs of the numerical estimates}\label{app:auxiliary}

\begin{proof}[Proof of Lemma~\ref{lem:binomial}]
Set $c_j=4^{-j}\binom{2j}j$. Its recurrence shows that
$c_j\sqrt j$ increases and $c_j\sqrt{2j+1}$ decreases.
It also implies $1/(2\sqrt j)\le c_j\le1/\sqrt{j+1}$ for $j\ge1$.
The largest binomial probability in dimension $d-1$ is
$c_{\lfloor d/2\rfloor}$, between $1/\sqrt{2d}$ and $\sqrt{2/d}$.
The gap $1-2w(d,t)$ consists of $t-1$ central probabilities in
dimension $d$ and contains a largest one. These observations give
\eqref{eq:explicitbinomial}.

Set $k=\lfloor(d-1)/2\rfloor$ and $r=(d-t)/2$. The ratio
$$
 \frac{\binom{d-1}r}{\binom{d-1}k}
 =\prod_{j=0}^{k-r-1}\frac{k-j}{d-k+j}
$$
is at least $1-t^2/(2d)$. Indeed each denominator is at least
$d/2$, and the sum of the factor deficits is at most
$2(k-r)^2/d$ for odd $d$, or $2(k-r)(k-r+1)/d$ for even $d$.
This proves \eqref{eq:centralw}.
Direct squaring verifies
$c_{64}\sqrt{128}>39/49$ and $c_{64}\sqrt{129}<4/5$.
The two monotonicities imply \eqref{eq:profile-upper-sharp} and,
for $d\ge256t^2$, the bound
$\sqrt d\,\theta(d,t)\ge(39/49)(1-2^{-9})>27/34$,
which is \eqref{eq:profile-lower-central}.

Each of the $t-1$ central probabilities in $1-2w(d,t)$ is at
least $\theta(d,t)$: its endpoint probability is
$d\theta(d,t)/(2(r+1))\ge\theta(d,t)$.
Let $d_0$ be the least admissible integer at least $256t^2+2$.
Then $d_0\le256t^2+3\le18^2t^2$.
By monotonicity of $w(\cdot,t)$, a consequence of
\eqref{eq:incrementa}, we have, for admissible $d\le256t^2+2$,
$$
 \frac12-w(d,t)\ge\frac{t-1}{2}\theta(d_0,t)
 >\frac{3(t-1)}{8\sqrt{d_0}}\ge\frac{t-1}{48t}\ge\frac1{96}.
$$
This proves \eqref{eq:closer-half-dimension}.
\end{proof}

\begin{proof}[Proof of Lemma~\ref{lem:numerical}]
Set $x=p+q-2\ge2$. We use the elementary bounds
\begin{equation}\label{eq:appendix-log-bounds}
 \ln2<\frac{52}{75},\qquad
 \frac{12}{11}<\ln3<\frac{11}{10},\qquad
 \ln5<\frac{29}{18}.
\end{equation}
Use $\ln a=2\sum_{j\ge0}r^{2j+1}/(2j+1)$, $r=(a-1)/(a+1)$,
whose tail after $k$ terms is less than
$2r^{2k+1}/((2k+1)(1-r^2))$, with $(a,k)=(2,2),(3,3),(5,4)$.

At $n=2^{15}x^3$, the left sides of \eqref{eq:initiallog} and
\eqref{eq:weak-bootstrap-number} are $88\sqrt{2x}/5$ and greater
than $100x$, respectively, while
$\ln(n+1)\le15\ln2+3\ln x+2^{-18}$.
Both $88\sqrt{2x}/5-(15/2)\ln x$ and $25x/4-3\ln x$ increase
for $x\ge2$. At two, by \eqref{eq:appendix-log-bounds}, we have
$176/5>47\ln2+1+5\cdot2^{-19}$ and
$25/2>18\ln2+2^{-18}$, proving the two assertions.
Differentiation extends them to every larger $n$.

Since $18^2/96=27/8$, we have $L\ge54\ln(n+1)$.
Also $n\ge2^{16}x^2$, $L\ge(2^{15}/96)x>256x$, and
$$
 m\ge\frac{n}{2x^{2/3}}\ge2^{14}x^{7/3},\qquad
 \frac Lm\le\frac1{48x^{4/3}}<\frac1{64},\qquad
 mL\ge\frac{n^2}{192x^{8/3}}\ge2^{22}x^{10/3},
$$
Together with $nL\ge2^{23}x^4\ge192^2x^2$, these prove all size bounds.

For \eqref{eq:coarse-error-number}, use $h\ge1$, $u\le n$ and take
logarithms. At $n=2^{15}x^3$, for $x\ge3$ it suffices that
\begin{equation}\label{eq:appendix-coarse-log}
 \frac{88\sqrt{2x}}5
 \ge\frac{21}{8}\ln(2^{15}x^3)+\ln3+\frac54\ln(16/3).
\end{equation}
The difference in \eqref{eq:appendix-coarse-log} increases for $x\ge3$;
at three the left side exceeds $42$ and the right side is less than $40$.
For $x=2$, the bound $u\le5n/8$ adds $(13/8)\ln(5/8)$ to the
right side. Substitution of \eqref{eq:appendix-log-bounds} then makes
it less than $176/5$, the left side at $n=2^{18}$.
Finally $11\sqrt n/(80x)-(21/8)\ln n$ increases throughout the
stated range, proving the assertion for all permitted $n$.
\end{proof}

\section{An inequality for log-concave sequences for Section~\ref{sec:exterior}}
\label{app:shadow-tools}

A positive sequence $q_0,\ldots,q_M$ is \emph{log-concave} if
$q_j^2\ge q_{j-1}q_{j+1}$ for $1\le j<M$; equivalently,
$q_{j+1}/q_j$ is nonincreasing. Multiplication by $\gamma^{-j}$,
where $\gamma>0$, preserves this property.

\begin{lemma}\label{lem:interval-coupling}
Let $q_0,\ldots,q_M$ be positive and log-concave,
$1\le h\le M$, and $\gamma\ge1$.
For $z_0,\ldots,z_M\ge0$, extend both sequences by zero outside
$\{0,\ldots,M\}$. Then
\begin{equation}\label{eq:interval-coupling}
 \sum_{j=0}^M q_{j+h}\max_{0\le i\le h}\gamma^iz_{j+i}
 \ge\frac{\sum_{j=0}^M q_{j+h}\gamma^{-j}}
           {\sum_{j=0}^M q_j\gamma^{-j}}
        \sum_{j=0}^M q_jz_j.
\end{equation}
\end{lemma}

\begin{remark*}
The probability coupling needed here is a special case of
Bartroff, Goldstein and I\c{s}lak~\cite[Corollary~2.1]{BGI}.
Their result states that an integer-valued random variable $K$ with
log-concave probabilities and support $\{0,\ldots,M\}$ can be
coupled to a random variable $K_h$ with the conditional distribution
of $K$ given $K\ge h$, so that $K\le K_h\le K+h$ almost surely.
Here a coupling means defining the variables on one probability
space while keeping their specified individual distributions.
The weighted inequality \eqref{eq:interval-coupling} follows as below.
\end{remark*}

\begin{proof}
Choose $K$ with
$$
 \PP(K=j)=\frac{q_j\gamma^{-j}}
                   {\sum_{a=0}^M q_a\gamma^{-a}},
 \qquad 0\le j\le M.
$$
These probabilities are log-concave and positive on
$\{0,\ldots,M\}$. By \cite[Corollary~2.1]{BGI}, there is a
$K_h$ on the same probability space satisfying
$$
 \PP(K_h=j)=\PP(K=j\mid K\ge h),
 \qquad K\le K_h\le K+h\quad\text{almost surely}.
$$
Put $J=K_h-h$. Then $J\le K\le J+h$ and
$$
 \PP(J=j)=\frac{q_{j+h}\gamma^{-j}}
                   {\sum_{a=0}^M q_{a+h}\gamma^{-a}},
 \qquad 0\le j\le M.
$$
Since $K$ belongs to the integer interval $[J,J+h]$,
$$
 \gamma^Kz_K\le
       \max_{0\le i\le h}\gamma^{J+i}z_{J+i}
 \quad\text{almost surely}.
$$
Taking expectations and substituting the two probability formulas
gives
$$
 \frac{\sum_{j=0}^M q_jz_j}{\sum_{j=0}^M q_j\gamma^{-j}}
 \le
 \frac{\sum_{j=0}^M q_{j+h}\max_{0\le i\le h}\gamma^iz_{j+i}}
      {\sum_{j=0}^M q_{j+h}\gamma^{-j}}.
$$
The denominators are positive, so this is
\eqref{eq:interval-coupling}.
\end{proof}

\vspace{10pt}
{\noindent \bf Acknowledgement.} Some of the results in this paper were obtained during the first author's
visit to Sichuan University at the invitation of the third author. The first author thanks the third author for the warm hospitality. The first author was supported by the National Natural Science Foundation of China (Grant No. 124B2019) and the third author was supported by the
National Natural Science Foundation of China (Grant No. 12622121).

\vspace{8pt}
{\bf\noindent AI declaration}: The
 work was done with the aid of generative AI tools. The authors had the rough proof strategies  and  generative AI tools filled in most of the  computational details, which was then checked and rewritten by the authors.

\end{document}